\documentclass[12pt]{amsart}

\usepackage{amsmath}
\usepackage{amssymb}
\usepackage{amstext}
\usepackage{amscd}
\usepackage{amsthm}

\usepackage{color}

\usepackage[colorlinks]{hyperref}
\hypersetup{linkcolor=blue, citecolor=red}

\newtheorem{teo}{Theorem}[section]
\newtheorem{prop}[teo]{Proposition}
\newtheorem{lem}[teo]{Lemma}
\newtheorem{cor}[teo]{Corollary}
\newtheorem{conj}[teo]{Conjecture}

\newtheorem{defini}[teo]{Definition}

\newtheorem{rem}[teo]{Remark}

\newcommand{\SL}{{\bf SL}}

\newcommand{\Lie}{{\rm Lie}}

\newcommand{\der}{{\rm der}}

\newcommand{\ad}{{\rm ad}}

\newcommand{\CC}{{\mathbb C}}
\newcommand{\RR}{{\mathbb R}}
\newcommand{\ZZ}{{\mathbb Z}}
\newcommand{\QQ}{{\mathbb Q}}
\newcommand{\NN}{{\mathbb N}}

\title[Convergence of measures on locally symmetric spaces]{Convergence of measures on compactifications of locally symmetric spaces}

\author[Daw\and Gorodnik\and Ullmo]{Christopher Daw\and Alexander Gorodnik\and Emmanuel Ullmo}

\begin{document}

\begin{abstract}
We conjecture that the set of homogeneous probability measures on the maximal Satake compactification of an arithmetic locally symmetric space $S=\Gamma\backslash G/K$ is compact. More precisely, given a sequence of homogeneous probability measures on $S$, we expect that any weak limit is homogeneous with support contained in precisely one of the boundary components (including $S$ itself).  We introduce several tools to study this conjecture and we prove it in a number of cases, including when $G=\SL_3(\RR)$ and $\Gamma=\SL_3(\ZZ)$.
\end{abstract}

\maketitle

\tableofcontents

\section{Introduction}

\subsection{Background and Motivations}
The study of sequences of measures invariant under unipotent flows has been a central theme in homogeneous dynamics, and the deep theorems obtained have had several important arithmetic applications. Prototypical in this respect is Margulis's proof of the Oppenheim Conjecture concerning the values of irrational indefinite quadratic forms at integral vectors \cite{margulis:forms}. Margulis obtained his result by characterising ${\bf \rm SO(2,1)}$-orbits on the homogeneous space $\SL(3,\ZZ)\backslash \SL(3,\RR)$.

More generally, let $\bf{G}$ be a semisimple algebraic group over $\QQ$ and let $G:={\bf G}(\RR)^+$. Let $\Gamma\subset {\bf G}(\QQ)$ be an arithmetic lattice (henceforth known as an arithmetic subgroup) contained in $G$ and let ${\bf H}_\RR$ be a real algebraic subgroup of ${\bf G}_{\RR}$ such that $H:={\bf H}_{\RR}(\RR)^+$ is generated by its one parameter unipotent subgroups. In her seminal works \cite{ratner:measure}, \cite{ratner:flows}, Ratner obtained a classification of the $H$-invariant ergodic measures on the homogeneous space $\Gamma\backslash G$ and proved that the closure of an $H$-orbit
is a homogeneous subspace of $\Gamma\backslash G$. 

For any algebraic $\QQ$-subgroup $\bf{H}$ of $\bf{G}$ without $\QQ$-rational 
characters, $\Gamma\backslash G$ admits
a canonical $H$-invariant probability measure $\mu_H$ with support 
$\Gamma\cap H\backslash H\subseteq \Gamma\backslash G$, where $H:={\bf H}(\RR)^+$. For any $g\in G$, we denote by 
$\mu_{H,g}$ the push-forward of $\mu_H$ by the right-multiplication-by-$g$ map. That is, 
$\mu_{H,g}$ is supported on $\Gamma\cap H\backslash Hg$. Such a measure is called {\it homogeneous} or {\it arithmetic}.

Mozes-Shah \cite{MS:ergodic} and Eskin-Mozes-Shah \cite{EMS:flows}, \cite{EMS:nondivergence}
started the study of weak convergence of sequences $(\mu_n)_{n\in \NN}=(\mu_{H_{n,g_n}})_{n\in \NN}$ of homogeneous measures associated with sequences $({\bf H}_n)_{n\in \NN}$ of subgroups of $\bf{G}$ and sequences
$(g_n)_{n\in \NN}$ of elements of $G$. 
With natural hypotheses on the $H_n$ and $g_n$, it follows from the work of Mozes-Shah \cite{MS:ergodic} that, if $\mu$ is a weak limit of $(\mu_n)_{n\in \NN}$ in the space of probability measures on $\Gamma\backslash G$, then $\mu$ is homogeneous itself. Furthermore, they showed that, if $\mu$ is a weak limit of $(\mu_n)_{n\in \NN}$ in the space of probability measures on the one-point-compactification $\Gamma\backslash G\cup \{\infty\}$ of $\Gamma\backslash G$, then $\mu$ is either a homogeneous probability measure on $\Gamma\backslash G$, or equal to the Dirac delta measure at infinity. In \cite{EMS:nondivergence}, building on the earlier work of Dani-Margulis \cite{DM:uniflows}, Eskin-Mozes-Shah proved a non-divergence criterion for sequences of homogeneous measures and, motivated by a counting problem for lattice points on homogeneous varieties, applied this to show in \cite{EMS:flows} that, when, for all $n\in\NN$, ${\bf H}_n={\bf H}$, for a fixed reductive subgroup ${\bf H}$ of ${\bf G}$ not contained in a proper $\QQ$-parabolic subgroup of $\bf G$, any weak limit $\mu$ of $\mu_n=\mu_{H,g_n}$ is homogeneous. All of these works relied on the fundamental results of Ratner. 

The aim of this paper is to study these questions for (arithmetic) locally symmetric spaces in the case when a sequence of homogeneous measures diverges. More precisely, we study weak limits of homogeneous measures on suitable compactifications of locally symmetric spaces.
 
Fix a maximal compact subgroup $K$ of $G$ and denote by $X=G/K$ the associated Riemannian symmetric space. Let $S=\Gamma\backslash X$ be the corresponding locally symmetric space and let
\begin{align*}
\pi: \Gamma\backslash G \longrightarrow S:\Gamma g\mapsto \Gamma g K
\end{align*}
be the projection map. A {\it homogeneous probability measure} on $S$ is defined as the push forward 
 $\pi_*(\mu_{H,g})$ of a homogeneous probability measure $\mu_{H,g}$ on $\Gamma\backslash G$, for any ${\bf H}$ of type $\mathcal{H}$ (see Section \ref{typeH}).
 
The key point is that locally symmetric spaces have natural compactifications of the form
 \begin{equation}
 \overline{\Gamma\backslash X}^S_{\rm max}= \Gamma\backslash X\cup\coprod_{\bf P} \Gamma_P\backslash X_P,
 \end{equation}
where ${\bf P}$ varies over a (finite) set of representatives for the $\Gamma$-conjugacy classes of $\QQ$-parabolic
subgroups of ${\bf G}$, and the boundary components $\Gamma_P\backslash X_P$ are themselves locally symmetric spaces. We will be mainly concerned with the maximal Satake compactification of $S$ in this text, but we also discuss the Baily-Borel compactification when $S$ is a hermitian locally symmetric space.

We make the following, seemingly natural conjecture, which we state in a more precise form in Section \ref{s3} (see Conjectures \ref{mainconj} and \ref{bbconj}).

\begin{conj}\label{introconj}
Suppose that $\mu$ is a probability measure on $\overline{\Gamma\backslash X}^S_{\rm max}$ equal to the weak limit of a sequence  $(\mu_{H_n,g_n})_{n\in\NN}$ of homogeneous probability measures on $S$. Then $\mu$ is a homogeneous probability measure supported on precisely one of the boundary components.

\end{conj}

The main purpose of this work is to discuss Conjecture \ref{introconj} and to establish it under additional restrictions. The novelty here is that we do not assume that ${\bf H}_n$ is not contained in a proper rational parabolic of $\bf G$, and we therefore need to study the behavior of sequences of homogeneous measures when the mass escapes at infinity.

The original motivation for this work concerned the case when $S$ is a hermitian locally symmetric space. Then the Baily-Borel compactification of $S$ realises $S$ as a quasi-projective algebraic variety and $S$ has an interpretation as a moduli space for interesting structures (often, abelian varieties with level structures and endomorphisms). The boundary components of the Baily-Borel compactification
of $S$ are themselves hermitian locally symmetric spaces. In this situation, the Andr\'e-Oort and Zilber-Pink conjectures
predict strong restrictions on the distribution of special (or weakly special) subvarieties of $S$, which are homogeneous subvarieties of $S$ also possessing hermitian locally symmetric structures. Several results on the equidistribution 
of sequences of measures associated with special subvarieties were obtained by Clozel and the third author \cite{CU:strong}, \cite{ullmo:NF}, and these played a central role in the first proof of the Andr\'e-Oort Conjecture under the generalised Riemann hypothesis by Klingler, Yafaev and the third author \cite{ky:andre-oort}, \cite{uy:andre-oort}. This paper deals principally with the convergence of measures on general locally symmetric spaces and their Satake compactifications, but we hope to discuss the Baily-Borel compactification, and possible applications, in a future work.

\subsection{Overview of the results}

Section \ref{s2} is mainly preliminary. We recall relevant results on root systems, parabolic subgroups and ergodic theory on homogeneous spaces. We make repeated use of the rational Langlands decomposition of $G$ associated with 
a parabolic subgroup ${\bf P}$ of ${\bf G}$ defined over $\QQ$, which is described in Section \ref{ratLang}. We recall here that this decomposition is of the form
\begin{align*}
G=N_{\bf P}A_{\bf P}M_{\bf P}K,
\end{align*}
where $N_{\bf P}$ is the unipotent radical of $P:={\bf P}(\RR)^+$, $A_{\bf P}$ is the identity component of a real algebraic split torus, $A_{\bf P} M_{\bf P}$ is a Levi subgroup of $P$ and $K$ a maximal compact subgroup of $G$.

In Section \ref{s3}, we recall definitions and  properties of the maximal Satake compactification of a locally symmetric space $S$, and
of the Baily-Borel compactification when $S$ is hermitian. We then formulate our main conjectures on general sequences of homogeneous measures for the maximal Satake (Conjecture \ref{mainconj}) and the Baily-Borel (Conjecture \ref{bbconj}) compactifications. Theorem
 \ref{maintobb} shows that the conjecture for the
Satake compactification implies the one for the Baily-Borel compactification.

In Sections \ref{s4} and \ref{s5}, we prove two convergence criteria for sequences 
$(\mu_n)_{n\in \NN}= (\mu_{H_n, g_n})_{n\in \NN}$ of homogeneous measures. Theorem \ref{criterionprop} gives a sufficient condition, in terms of $(\phi,\chi)$-functions, as introduced by Borel \cite{borel:arithmetic-groups}, Section 14, under which $\{\mu_n\}_{n\in\NN}$ is sequentially compact. Theorem \ref{convergencecriterion} gives a sufficient condition under which $\mu_n$ converges
to a homogeneous measure $\mu$ on a boundary component $\Gamma_P\backslash X_P$ of the maximal Satake compactification of $\Gamma\backslash X$. These two results are crucial in the rest of the paper and are the main tools at our disposal. In order to use these criteria, it is necessary to
\begin{itemize}
\item understand the set of rational parabolic subgroups of ${\bf G}$ containing ${\bf H}_n$;
\item for each parabolic subgroup ${\bf P}$ containing ${\bf H}_n$, compute the rational Langlands decomposition
\begin{align*}
g_n=u_{n}a_{n}m_{n}k_{n}\in G=N_{\bf P}A_{\bf P}M_{\bf P}K;
\end{align*}
 \item for any $\alpha$ in a set of simple roots for the action of $A_{\bf P}$ on $N_{\bf P}$, understand the behavior of 
 $\alpha(a_{n})$ as $n\rightarrow\infty$ (positivity, boundedness, convergence to $\infty$).
\end{itemize}

In Sections \ref{s6} and \ref{s7}, we prove Conjecture \ref{mainconj} in full generality when the $\QQ$-rank of ${\bf G}$ is $0$ or $1$. Then, for any ${\bf G}$, Theorem \ref{levi} establishes Conjecture \ref{mainconj} when, for all $n\in\NN$, ${\bf H}={\bf H}_n$ is the semisimple non-compact part of a Levi subgroup of a maximal parabolic subgroup over $\QQ$. 

Theorem \ref{teo8.1}, which is one of the main results of the paper, establishes Conjecture \ref{mainconj} in the case when ${\bf G}={\bf \SL}_3$ and $\Gamma=\SL(3,\ZZ)$. The complexity of the general problem can already be seen from the various cases we have to face in this situation. Theorem \ref{teo9.1} establishes Conjecture \ref{mainconj} when, for $r\in\NN$, ${\bf G}={\bf \SL}_2^r$ and $\Gamma=\SL(2,\ZZ)^r$. This, of course, is an instance when $S$ is hermitian.

Theorem \ref{unip} establishes Conjecture \ref{mainconj} when, for each $n\in \NN$, ${\bf H}_n$
is equal to the unipotent radical $\bf N$ of a minimal parabolic subgroup ${\bf P}_0$ of ${\bf G}$. In this case, a weak limit $\mu$ of a sequence $(\mu_{N,g_n})_{n\in\NN}$
can be a homogeneous measure supported on any boundary component $\Gamma_P\backslash X_P$. The proof is constructive and explains, in terms of the rational Langlands decompositions of the $g_n$ relative to ${\bf P}_0$, on which boundary component $\mu$ is supported.

In Section \ref{s11}, we recall some basic properties of the Tits building $\mathcal{B}$ of ${\bf G}$. In particular, we discuss the notion of a {\it Levi sphere}, as introduced by Serre \cite{serre:reduct}, which is a sub-simplicial complex  $\mathcal{S}$ of 
$\mathcal{B}$ contained in an apartment of $\mathcal{B}$. The simplexes of $\mathcal{S}$ parametrise the rational parabolic subgroups of 
${\bf G}$ containing a fixed Levi subgroup of some parabolic subgroup. This notion is used in Section \ref{s13}
to study the conjecture when we translate subgroups $H_n$ of $M_{\bf P}$, for some parabolic subgroup ${\bf P}$, by elements $a_n\in A_{\bf P}$. We show that $\mathcal{S}$
can be described as the unit sphere in the Lie algebra $\mathfrak{a}_{\bf P}$ of $A_{\bf P}$. We then find a simplex 
in $\mathcal{S}$ corresponding to a parabolic subgroup ${\bf Q}$ such that $Q=N_{\bf Q} A_{\bf P}M_{\bf P}$ and the roots of $A_{\bf P}$
in $N_{\bf Q}$ take positive values on $\exp^{-1}(a_n)\in \mathfrak{a}_P$. This allows us to apply Theorem \ref{convergencecriterion}.

\subsection*{Acknowledgements}

The first author would like to thank the EPSRC and the Institut des Hautes Etudes Scientifiques (IHES) for granting him a William Hodge fellowship, during which he first began working on this topic. He would like to thank the EPSRC again, as well as Jonathan Pila, for the opportunity to be part of their project Model theory, functional transcendence, and diophantine geometry (EPSRC reference EP/N008359/1). He would like to thank the University of Reading for its ongoing support.

The first and second authors would both like to thank the IHES for an enjoyable research visit in August 2017.

\section{Preliminaries}\label{s2}

\subsection{Borel probability measures}\label{s2.1}
Let $S$ be a metric space and let $\Sigma$ be its Borel $\sigma$-algebra. By a {\it Borel probability measure on $S$}, we mean a Borel probability measure on $\Sigma$. We let $\mathcal{P}(S)$ denote the space of all Borel probability measures on $S$. We say that a sequence $(\mu_n)_{n\in\NN}$ in $\mathcal{P}(S)$ {\it converges (weakly)} to $\mu\in\mathcal{P}(S)$ if we have
\begin{align*}
\int_{S}f\ d\mu_n\rightarrow\int_{S}f\ d\mu,\text{ as }n\rightarrow\infty,
\end{align*}
for all bounded continuous functions $f$ on $S$.

\subsection{Algebraic groups}\label{s2.2}
By an {\it algebraic group}, we refer to a linear algebraic group defined over $\QQ$ and by an {\it algebraic subgroup of $\bf G$} we again refer to an algebraic subgroup defined over $\QQ$. We will use boldface letters to denote algebraic groups (which, again, are always defined over $\QQ$). If $\bf G$ is an algebraic group, we will denote its radical by $\bf R_G$ and its unipotent radical by ${\bf N}_{\bf G}$. We will write ${\bf G}^\circ$ for the (Zariski) connected component of $\bf G$ containing the identity. We will denote the Lie algebra of $\bf G$ by the corresponding mathfrak letter $\mathfrak{g}$, and we will denote the (topological) connected component of ${\bf G}(\RR)$ containing the identity by the corresponding Roman letter $G$. We will denote by $G_\QQ$ the intersection ${\bf G}(\QQ)\cap G$. We will retain any subscripts or superscripts in these notations. If $\bf M$ and $\bf A$ are algebraic subgroups of $\bf G$, we will write ${\bf Z_M}({\bf A})$ for the centraliser of $\bf A$ in $\bf M$ and ${\bf N_M(A)}$ for the normaliser of $\bf A$ in $\bf M$.

\subsection{Parabolic subgroups}\label{s2.3}
 A {\it parabolic subgroup} $\bf P$ of a connected algebraic group $\bf G$ is an algebraic subgroup such that the quotient of ${\bf G}$ by $\bf P$ is a projective algebraic variety. In particular, $\bf G$ is a parabolic subgroup of itself. However, by a {\it maximal} parabolic subgroup, we refer to a maximal {\it proper} parabolic subgroup. Note that $\bf R_G$ is contained in every parabolic subgroup of $\bf G$.

\begin{lem}[Cf. \cite{bt:groupes}, Proposition 4.4]\label{int}
Let ${\bf P}$ be a parabolic subgroups of $\bf G$. Then ${\bf N_G}({\bf P})={\bf P}$.

If ${\bf Q}$ is a another parabolic subgroup of $\bf G$, then $({\bf P}\cap{\bf Q}){\bf N_P}$ is a parabolic subgroup of $\bf G$, which is equal to $\bf P$ if and only if $\bf Q$ contains a Levi subgroup of $\bf P$.
\end{lem}

\begin{cor}\label{norm}
Let $\bf P$ be a parabolic subgroup of $\bf G$. Then ${\bf N_G}({\bf N_P})={\bf P}$.
\end{cor}

\begin{proof}
Let $g\in{\bf N_G}({\bf N_P})(\QQ)$. Then ${\bf Q}:=g{\bf P}g^{-1}$ is a parabolic subgroup of ${\bf G}$ containing ${\bf N_P}$. Therefore, ${\bf P}\cap{\bf Q}$ contains the parabolic subgroup $({\bf P}\cap{\bf Q}){\bf N_P}$ of $\bf G$, and is, therefore, a parabolic subgroup of ${\bf G}$ itself. By \cite{bt:groupes}, Section 4.3, we have ${\bf Q}={\bf P}$ and so $g\in{\bf P}(\QQ)$. The result follows from the fact that ${\bf N_G}({\bf N_P})(\QQ)$ is Zariski dense in ${\bf N_G}({\bf N_P})$.
\end{proof}

\subsection{Cartan involutions}\label{s2.4}
Let $\bf G$ be a reductive algebraic group and let $K$ be a maximal compact subgroup of $G$. Then there exists a unique involution $\theta$ on $G$ such that $K$ is the fixed point set of $\theta$. We refer to $\theta$ as the {\it Cartan involution of $G$ associated with $K$}.

\subsection{Boundary symmetric spaces}\label{s2.5}
Let $\bf G$ be a semisimple algebraic group and let $K$ be a maximal compact subgroup of $G$. Let $\bf P$ be a parabolic subgroup of $\bf G$. As in \cite{BJ:compactifications}, (I.1.10), we have the {\it real Langlands decomposition (with respect to $K$)}
\begin{align*}
P=N_PM_PA_P,
\end{align*}
where $L_P:=M_PA_P$ is the unique Levi subgroup of $P$ such that $K_P:=L_P\cap K=M_P\cap K$ is a maximal compact subgroup of $L_P$, and $A_P$ is the maximal split torus in the centre of $L_P$. We denote by $X_P$ the {\it boundary symmetric space} $M_P/K_P$, on which $P$ acts through its projection on to $M_P$.

\subsection{Rational Langlands decomposition}\label{ratLang}

Let ${\bf G}$ be a connected algebraic group and let $K$ be a maximal compact subgroup of $G$. Let $\bf P$ be a parabolic subgroup of $\bf G$. As in \cite{BJ:compactifications}, (III.1.3), we have the {\it rational Langlands decomposition (with respect to $K$)}
\begin{align*}
P=N_{\bf P}M_{\bf P}A_{\bf P}.
\end{align*}
We let $\Phi(P,A_{\bf P})$ denote the set of characters of $A_{\bf P}$ occuring in its action on $\mathfrak{n}_P$. Since $G=PK$, the rational Langlands decomposition of $P$ yields
\begin{align*}
G=N_{\bf P}M_{\bf P}A_{\bf P}K.
\end{align*}
In particular, if $g\in G$, we can write $g$ as 
\begin{align*}
g=nmak\in N_{\bf P}M_{\bf P}A_{\bf P}K,
\end{align*}
and we denote the (uniquely determined) $A_{\bf P}$-component $a_{{\bf P},K}(g)$.

We remind the reader that the groups $A_{\bf P}$ and $M_{\bf P}$ are not necessarily associated with algebraic groups defined over $\QQ$. However, by \cite{BJ:compactifications}, Proposition III.1.11, there exists an $n\in N_{\bf P}$ such that $nM_{\bf P}n^{-1}$ and $nA_{\bf P}n^{-1}$ are associated with algebraic groups defined over $\QQ$. In particular, the product $N_{\bf P}M_{\bf P}$ is associated with a connected algebraic group over $\QQ$, which we always denote $\bf H_P$. Clearly, ${\bf H_P}$ is a group with no rational characters. Note that the rational Langlands decomposition with respect to $nKn^{-1}$ is
\begin{align*}
P=N_{\bf P}\cdot nM_{\bf P}n^{-1}\cdot nA_{\bf P}n^{-1},
\end{align*}
from which it follows that $\bf H_P$ depends only on $\bf P$.

\subsection{Standard parabolic subgroups}\label{standard}

Let $\bf G$ be a reductive algebraic group and let $\bf A$ be a maximal split subtorus of $\bf G$. The non-trivial characters of ${\bf A}$ that intervene in the adjoint representation of ${\bf G}$ restricted to ${\bf A}$ are known as the {\it $\QQ$-roots of ${\bf G}$ with respect to ${\bf A}$}.

Let ${\bf P}_0$ be a minimal parabolic subgroup of $\bf G$ containing $\bf A$. We let $\Phi({\bf P}_0,{\bf A})$ denote the set of characters of $\bf A$ occurring in its action on $\mathfrak{n}$, where ${\bf N}:={\bf N}_{{\bf P}_0}$. As explained in \cite{BJ:compactifications}, III.1.7, $\Phi({\bf P}_0,{\bf A})$ contains a unique subset $\Delta:=\Delta({\bf P}_0,{\bf A})$ such that every element of $\Phi({\bf P}_0,{\bf A})$ is a linear combination, with non-negative integer coefficients, of elements belonging to $\Delta$. On the other hand, ${\bf P}_0$ is determined by $\bf A$ and $\Delta$. We refer to $\Delta$ as a set of {\it simple $\QQ$-roots of $\bf G$ with respect to $\bf A$}.

For a subset $I\subseteq\Delta$, we define the subtorus
\begin{align*}
{\bf A}_I:=(\cap_{\alpha\in I}\ker\alpha)^{\circ}
\end{align*} 
of $\bf A$. Then the subgroup ${\bf P}_I$ of $\bf G$ generated by ${\bf Z}_{\bf G}({\bf A}_I)$ and $\bf N$ is a parabolic subgroup of $\bf G$. We refer to ${\bf P}_I$ as a {\it standard parabolic subgroup} of $\bf G$. Every parabolic subgroup of $\bf G$ containing ${\bf P}_0$ is equal to ${\bf P}_I$ for some uniquely determined subset $I\subseteq\Delta$. For ease of notation, when $I=\Delta\setminus\{\alpha\}$, for some $\alpha\in\Delta$, we will write ${\bf P}_\alpha$ instead of ${\bf P}_{\Delta\setminus\{\alpha\}}$. We will use the following lemma in Section \ref{s10}.

\begin{lem}\label{mustbestandard}
Let $\bf P$ be a parabolic subgroup of $\bf G$ containing $\bf N$. Then ${\bf P}$ contains ${\bf P}_0$. That is, ${\bf P}$ is a standard parabolic subgroup of $\bf G$.
\end{lem}

\begin{proof}
By assumption $\bf N$ is contained in ${\bf P}\cap{\bf P}_0$, and, hence, ${\bf Q}:=({\bf P}\cap{\bf P}_0){\bf N}$, which, by Lemma \ref{int}, is a parabolic subgroup of $\bf G$. However, ${\bf Q}$ is contained in ${\bf P}_0$, which is minimal. Hence, ${\bf Q}={\bf P}_0$, which implies that ${\bf P}\cap{\bf P}_0={\bf P}_0$ and we conclude that $\bf P$ contains ${\bf P}_0$.
\end{proof}

Let $K$ be a maximal compact subgroup of $G$ such that $A$ is invariant under the Cartan involution of $G$ associated with $K$. Then, as in \cite{DM:uniflows}, Section 1, ${\bf Z}_{\bf G}({\bf A}_I)$ is the Levi subgroup of $\bf P$ appearing in the rational Langlands decomposition of $P$ with respect to $K$. Note that ${\bf A}_I$ is the maximal split subtorus of the centre of ${\bf Z_G(A}_I)$ and we can write ${\bf Z_G(A}_I)$ as an almost direct product ${\bf M}_I{\bf A}_I$, where ${\bf M}_I$ is a reductive group with no rational characters. The rational Langlands decomposition is then
\begin{align*}
P_I=N_IM_IA_I,
\end{align*}
where ${\bf N}_I:={\bf N}_{{\bf P}_I}$. We will also write ${\bf H}_I:={\bf N}_I{\bf M}_I$. For ease of notation, when $I=\Delta\setminus\{\alpha\}$, for some $\alpha\in\Delta$, we will write ${\bf A}_\alpha$, ${\bf M}_\alpha$, ${\bf N}_\alpha$ and ${\bf H}_\alpha$ instead of ${\bf A}_I$, ${\bf M}_I$, ${\bf N}_I$ and ${\bf H}_I$, respectively.

\subsection{Root systems}\label{weyl}

Consider the situation described in Section \ref{standard}. Let $X^*({\bf A})$ denote the character module of $\bf A$, and let $X^*({\bf A})_{\QQ}$ denote the $\QQ$-vector space $X^*({\bf A})\otimes_{\ZZ}\QQ$. (We will also later refer to the cocharacter module $X_*({\bf A})$ of $\bf A$.) Fix a non-degenerate scalar product $(\cdot,\cdot)$ on $X^*({\bf A})_{\QQ}$ that is invariant under the action of ${\bf N}_{\bf G}({\bf A})(\QQ)$. Then the $\QQ$-roots of ${\bf G}$ with respect to ${\bf A}$  equipped with the inner product $(\cdot,\cdot)$ constitute a {\it root system} in $X^*({\bf A})_{\QQ}$. We refer the reader to \cite{springer:redgroups}, Section 3.5 for further details.

Let $\Delta$ denote a set of simple $\QQ$-roots of $\bf G$ with respect to $\bf A$ and fix a subset $I\subseteq\Delta$. Then ${\bf A}^I:=({\bf A}\cap{\bf M}_I)^{\circ}$ is a maximal split subtorus of ${\bf M}_I$ and the elements of $I$ restrict to a set of simple $\QQ$-roots of ${\bf M}_I$ with respect to ${\bf A}^I$. Furthermore, ${\bf A}$ is equal to the almost direct product ${\bf A}^I{\bf A}_I$, from which it follows that ${\bf N}_{{\bf M}_I}({\bf A}^I)$ is a subgroup of ${\bf N}_{\bf G}({\bf A})$. For ease of notation, when $I=\Delta\setminus\{\alpha\}$, for some $\alpha\in\Delta$, we will write ${\bf A}^\alpha$ instead of ${\bf A}^I$.

The isogeny ${\bf A}^{I}\times{\bf A}_{I}\rightarrow{\bf A}:(a,b)\mapsto ab$ yields an identification 
\begin{align}\label{directprod1}
X^*({\bf A})_\QQ=X^*({\bf A}^{I})_\QQ\oplus X^*({\bf A}_{I})_\QQ
\end{align}
such that the projection $\pi_1$ (respectively, $\pi_2$) on to the first (respectively, second) factor is given by restricting to the corresponding subtorus. It follows that the restriction of $(\cdot,\cdot)$ to $X^*({\bf A}^{I})_\QQ$ is a non-degenerate scalar product that is invariant with respect to the action of ${\bf N}_{{\bf M}_{I}}({\bf A}^{I})(\QQ)$.

\begin{lem}\label{orth}
The decomposition (\ref{directprod1}) is orthogonal with respect to $(\cdot,\cdot)$. 
\end{lem}

\begin{proof}
Note that the elements of $I$ restrict to a basis of $X^*({\bf A}^{I})_\QQ$ and are trivial on ${\bf A}_{I}$. Choose any $\beta\in X^*({\bf A}_I)_{\QQ}$, and any $\alpha\in I$, and let $w\in {\bf N}_{{\bf M}_I}({\bf A}^I)(\QQ)$ be an element such that $w(\alpha)=-\alpha$. Since $w\in {\bf M}_I$, we have $w(\beta)=\beta$. Therefore,
\begin{align*}
(\alpha,\beta)=(w(\alpha),w(\beta))=(-\alpha,\beta)=-(\alpha,\beta)=0,
\end{align*}
which proves the claim.
\end{proof}

\subsection{Quasi-fundamental weights}\label{s2.9}
Consider the situation described in Section \ref{weyl}.
A set of {\it quasi-fundamental weights} in $X^*({\bf A})_\QQ$ is a set of elements $\chi_{\alpha}$, one for each $\alpha\in\Delta$, such that
\begin{align*}
(\chi_\alpha,\beta)=d_\alpha\cdot\delta_{\alpha\beta}\text{ for all }\alpha,\beta\in\Delta,
\end{align*}
where $d_\alpha\in\QQ_{>0}$ for all $\alpha\in\Delta$. We use the prefix {\it quasi-} to emphasise that we place no (further) restrictions on the $d_\alpha$. 

\begin{lem}\label{nonneg}
Let $\{\chi_\alpha\}_{\alpha\in\Delta}$ denote a set of quasi-fundamental weights in $X^*({\bf A})_\QQ$. Then, as a linear combination of the $\alpha\in\Delta$, each $\chi_\alpha$ has positive coefficients.
\end{lem}

\begin{proof}
The coefficients in question are, up to scaling, simply those of the inverse of the so-called {\it Cartan matrix}, which always has positive coefficients (see, for example, \cite{WZ:cartan}, Section 2.1).
\end{proof}

\begin{lem}\label{restr}
Let $\{\chi_\alpha\}_{\alpha\in\Delta}$ denote a set of quasi-fundamental weights in $X^*({\bf A})_\QQ$. Then the restrictions of the $\chi_\alpha$ for $\alpha\in I$ constitute a set of quasi-fundamental weights in $X^*({\bf A}^{I})_\QQ$ with respect to the restriction of $(\cdot,\cdot)$. 
\end{lem}

\begin{proof}
Let $\alpha,\beta\in I$. Then
\begin{align*}
d_\alpha\cdot\delta_{\alpha\beta}=(\chi_\alpha,\beta)=(\pi_1(\chi_\alpha)+\pi_2(\chi_\beta),\pi_1(\beta))=(\pi_1(\chi_\alpha),\pi_1(\beta)),
\end{align*}
where the second equality follows from the fact that $\pi_2(\beta)=0$ for any $\beta\in I$, and the third equality follows from Lemma \ref{orth}.
\end{proof}

Finally, we make an observation.

\begin{lem}\label{warning}
Let $\beta\in\Delta\setminus I$. Then, as a linear combination of the $\QQ$-simple roots $\pi_1(\alpha)$, for $\alpha\in I$, the restriction $\pi_1(\beta)$ has non-positive coefficients.
\end{lem}

\begin{proof}
Recall the basic fact that the scalar product of any two distinct simple roots is non-positive. Therefore, for any $\alpha\in I$,
\begin{align*}
(\pi_1(\alpha),\pi_1(\beta))=(\alpha,\beta)
\end{align*}
is non-positive, where we use the fact that $\pi_2(\alpha)=0$. Therefore, if we let $\{\chi_\alpha\}_{\alpha\in I}$ denote a set of quasi-fundamental weights in $X^*({\bf A}^{I})_\QQ$, then $\pi_1(\beta)$, as a linear combination of the $\chi_\alpha$, has non-positive coefficients. Hence, the result follows from Lemma \ref{nonneg}.
\end{proof}

\subsection{Groups of type $\mathcal{H}$}\label{typeH}
We say that an algebraic group $\bf G$ is {\it of type $\mathcal{H}$} if ${\bf R}_{\bf G}$ is unipotent and the quotient of $\bf G$ by ${\bf R}_{\bf G}$ is an almost direct product of almost $\QQ$-simple algebraic groups whose underlying real Lie groups are non-compact. In particular, an algebraic group of type $\mathcal{H}$ has no rational characters.

\subsection{Probability measures on homogeneous spaces}\label{measures}
Let $\bf G$ denote an algebraic group and let $\Gamma$ denote an arithmetic subgroup of ${\bf G}(\QQ)$ contained in $G$. We will henceforth refer to such a group as {\it arithmetic subgroup of $G_\QQ$}. If $\bf H$ is a connected algebraic subgroup of $\bf G$ possessing no rational characters, then there is a unique Haar measure on $H$ such that its pushforward $\mu$ to $\Gamma\backslash G$ is a Borel probability measure on $\Gamma\backslash G$. For $g\in G$, we refer to the pushforward of $\mu$ under the right-multiplication-by-$g$ map as the {\it homogeneous probability measure on $\Gamma\backslash G$ associated with ${\bf H}$ and $g$}.

\begin{rem}\label{freegamma}
It is clear that, for any $\gamma\in\Gamma$, the homogeneous probability measure on $\Gamma\backslash G$ associated with ${\bf H}$ and $g$ is equal to the homogenous probability measure on $\Gamma\backslash G$ associated with $\gamma{\bf H}\gamma^{-1}$ and $\gamma g$. 
\end{rem}

The following well-known fact summarises our heavy reliance on the fundamental results of Ratner \cite{ratner:measure} and of Mozes and Shah \cite{MS:ergodic}.

\begin{teo}\label{ratner}
For each $n\in\NN$, let ${\bf H}_n$ be a connected algebraic subgroup of $\bf G$ of type $\mathcal{H}$, let $g_n\in G$ and let $\mu_n$ be the homogeneous probability measure on $\Gamma\backslash G$ associated with ${\bf H}_n$ and $g_n$. Assume that $(\mu_n)_{n\in\NN}$ converges to $\mu\in\mathcal{P}(\Gamma\backslash G)$. Then $\mu$ is the homogeneous probability measure on $\Gamma\backslash G$ associated with a connected algebraic subgroup $\bf H$ of $\bf G$ of type $\mathcal{H}$ and an element $g\in G$, and, furthermore, ${\bf H}_n$ is contained in $\bf H$ for all $n$ large enough.
\end{teo}

We give a brief summary of the arguments. 

\begin{proof}[Proof of Theorem \ref{ratner}]
By \cite{CU:measures}, Lemme 3.1, for every $n\in\NN$, the subgroup of $g^{-1}_nH_ng_n$ generated by the unipotent one-parameter subgroups of $G$ contained in $g^{-1}_nH_ng_n$ acts ergodically on $\Gamma\backslash G$ with respect to $\mu_n$. By \cite{MS:ergodic}, Corollary 1.1, we conclude that the group generated by the unipotent one-parameter subgroups of $G$ contained in the invariance group of $\mu$ acts ergodically on $\Gamma\backslash G$ with respect to $\mu$. By \cite{ratner:measure}, the support of $\mu$ is a closed orbit of its invariance group. Therefore, the first claim follows from \cite{CU:measures}, Lemme 3.2. The second claim follows from \cite{MS:ergodic}, Theorem 1.1 (2) and \cite{CU:measures}, Lemme 3.2. 
\end{proof}

\section{Formulating the conjectures}\label{s3}

\subsection{The maximal Satake compactification}\label{sat}

Let $\bf G$ be a semisimple algebraic group of adjoint type and let $K$ be a maximal compact subgroup of $G$. Denote by $X$ the symmetric space $G/K$ and let $\Gamma$ denote an arithmetic subgroup of ${\bf G}(\QQ)$ contained in $G$. We let
\begin{align*}
_{\QQ}\overline{X}^S_{\rm max}:=\coprod_{\bf P}X_P,
\end{align*}  
where $\bf P$ varies over the (rational) parabolic subgroups of $\bf G$. We endow $_{\QQ}\overline{X}^S_{\rm max}$ with the topology defined in \cite{BJ:compactifications}, III.11.2. Then, by \cite{BJ:compactifications}, Proposition III.11.7, the action of $G_\QQ$ on $X$ extends to a continuous action on $_{\QQ}\overline{X}^S_{\rm max}$ and, by \cite{BJ:compactifications}, Theorem III.11.9, the quotient
\begin{align*}
\overline{\Gamma\backslash X}^S_{\rm max}:=\Gamma\backslash_{\QQ}\overline{X}^S_{\rm max},
\end{align*}
endowed with the quotient topology, is a compact Hausdorff space, inside of which $\Gamma\backslash X$ is a dense open subset. We refer to $\overline{\Gamma\backslash X}^S_{\rm max}$ as the {\it maximal Satake compactification} of $\Gamma\backslash X$. 

For any parabolic subgroup $\bf P$ of $\bf G$, we will denote by $\Gamma_P:=\Gamma\cap P$. Then, if $\mathcal{C}$ is any set of representatives for the (rational) parabolic subgroup of $\bf G$ modulo $\Gamma$-conjugation, the maximal Satake compactification $\overline{\Gamma\backslash X}^S_{\rm max}$ is equal to the disjoint union of the $\Gamma_P\backslash X_P$, with ${\bf P}$ varying over the members of $\mathcal{C}$.

\subsection{Main conjecture}\label{main}

Consider the situation described in Section \ref{sat}. If $\bf H$ is a connected algebraic subgroup of $\bf G$ of type $\mathcal{H}$ and $g\in G$, the homogeneous probability measure on $\Gamma\backslash G$ associated with ${\bf H}$ and $g$ pushes forward to $\overline{\Gamma\backslash X}^S_{\rm max}$ under the natural maps
\begin{align*}
\Gamma\backslash G\rightarrow\Gamma\backslash X\rightarrow\overline{\Gamma\backslash X}^S_{\rm max}.
\end{align*} 
We refer to this probability measure as the {\it homogeneous probability measure on $\overline{\Gamma\backslash X}^S_{\rm max}$ associated with $\bf H$ and $g$}. Similarly, if $\bf P$ is a parabolic subgroup of $\bf G$, $\bf H$ is a subgroup of ${\bf P}$ of type $\mathcal{H}$ and $g\in P$, we can define the {\it homogeneous probability measure on $\overline{\Gamma\backslash X}^S_{\rm max}$ associated with $\bf P$, $\bf H$ and $g$} in precisely the same way via the natural maps
\begin{align*}
\Gamma_P\backslash P\rightarrow\Gamma_P\backslash X_P\rightarrow\overline{\Gamma\backslash X}^S_{\rm max}.
\end{align*}

The following conjecture is a more precise version of Conjecture \ref{introconj} in this setting, and is the main statement that we will endeavour to prove in certain cases.

\begin{conj}\label{mainconj}
For each $n\in\NN$, let ${\bf H}_n$ be a connected algebraic subgroup of $\bf G$ of type $\mathcal{H}$, let $g_n\in G$ and let $\mu_n$ be the homogeneous probability measure on $\overline{\Gamma\backslash X}^S_{\rm max}$ associated with ${\bf H}_n$ and $g_n$. Suppose that $(\mu_n)_{n\in\NN}$ converges to a limit $\mu\in\mathcal{P}(\overline{\Gamma\backslash X}^S_{\rm max})$. Then $\mu$ is homogeneous.

Furthermore, if $\mu$ is associated with a parabolic subgroup ${\bf P}$ of $\bf G$, a connected algebraic subgroup $\bf H$ of ${\bf P}$ of type $\mathcal{H}$ and an element $g\in P$, then ${\bf H}_n$ is contained in $\bf H$ for $n$ large enough.
\end{conj}

Now consider another maximal compact subgroup $gKg^{-1}$ of $G$, for some $g\in G$ (recall that they are all of this form). It is straightforward to verify that we obtain a homeomorphism between the maximal Satake compactifications of $\Gamma\backslash X$ corresponding to $K$ and $gKg^{-1}$. In particular, Conjecture \ref{mainconj} is equivalent to the same statement in which $K$ is replaced with $gKg^{-1}$ and the $g_n$ are replaced with $g_ng^{-1}$.

Similarly, for a fixed $c\in G_\QQ$, we obtain a homeomorphism
\begin{align*}
\Gamma\backslash_{\QQ}\overline{X}^S_{\rm max}\rightarrow (c\Gamma c^{-1})\backslash_{\QQ}\overline{X}^S_{\rm max}
\end{align*}
of compactifications induced by the homeomorphism $x\mapsto cx$ on $_{\QQ}\overline{X}^S_{\rm max}$ (recall that the action is continuous). It follows that Conjecture \ref{mainconj} is equivalent to the same statement in which we replace $\Gamma$ with $c\Gamma c^{-1}$, and we replace the ${\bf H}_n$ with $c{\bf H}_nc^{-1}$ and the $g_n$ with $cg_n$.

Nontheless, despite the aformentioned observations, we are unable to provide an argument that Conjecture \ref{mainconj} is independent of the choice of $\Gamma$. Of course, Conjecture \ref{mainconj} for $\Gamma$ immediately implies Conjecture \ref{mainconj} for any arithmetic subgroup containing $\Gamma$. However, it is not clear that Conjecture \ref{mainconj} for $\Gamma$ implies Conjecture \ref{mainconj} for an arithmetic subgroup contained in $\Gamma$. Largely speaking, our arguments do not rely on the specific choice of $\Gamma$, though we do make use of the fact that $\Gamma={\bf SL}_3(\ZZ)$ in Section \ref{s8}, for example.

\subsection{Baily-Borel compactification}\label{bor}
Let $({\bf G},{\bf X})$ denote a Shimura datum, where $\bf G$ is a semisimple algebraic group of adjoint type, and let $X$ denote a connected component of $\bf X$. Let $K$ be a maximal compact subgroup of $G$ and identify $X$ with $G/K$. Let $\Gamma$ denote an arithmetic subgroup of $G_\QQ$.

Via the {\it Harish-Chandra realization}, we consider $X$ as a bounded domain in $\CC^N$ for some $N\in\NN$, and we let $\bar{X}$ denote the closure of $X$ therein. Then, as in \cite{BB:compactification}, Section 1.5, $\bar{X}$ decomposes into a disjoint union of boundary components and we let $X^*$ denote the union of the {\it rational} boundary components, as defined in \cite{BB:compactification}, Section 3.5. By \cite{BB:compactification}, Section 1.4, the action of $G$ on $X$ extends to a continuous action of $G$ on $\bar{X}$ and, by \cite{BB:compactification}, Section 4.8, this restricts to an action of $G_\QQ$ on $X^*$. 

We equip $X^*$ with the {\it Satake topology}, described in \cite{BB:compactification}, Theorem 4.9. For this topology, the action of $G_\QQ$ is continuous and, by \cite{BB:compactification}, Corollary 4.11, the quotient
\begin{align*}
\overline{\Gamma\backslash X}^{BB}:=\Gamma\backslash X^*
\end{align*}
endowed with the quotient topology, is a compact Hausdorff space, inside of which $\Gamma\backslash X$ is a dense open subset. In fact, the main result of \cite{BB:compactification} is that $\overline{\Gamma\backslash X}^{BB}$ possesses the structure of a complex projective variety. We refer to it as the {\it Baily-Borel compactification} of $\Gamma\backslash X$. In this case, $\overline{\Gamma\backslash X}^{BB}$ is the disjoint union of $\Gamma\backslash X$ and boundary components corresponding to $\Gamma$-conjugacy classes of maximal parabolic subgroups of $\bf G$.

Note that, for any Shimura datum $({\bf G},{\bf X})$, if $X$ is a connected component of $\bf X$, then the action of $G$ on $X$ factors through $G^{\ad}$, where ${\bf G}^{\rm ad}$ denotes the quotient of ${\bf G}$ by its centre. Furthermore, by \cite{milne:shimura}, Proposition 3.2, the image in ${\bf G}^{\rm ad}(\QQ)$ of an arithmetic subgroup of ${\bf G}(\QQ)$ is an arithmetic subgroup. Therefore, any connected component of any Shimura variety is accounted for in our description.

\subsection{Relationship between the compactifications}\label{rel}

Consider the situation described in Section \ref{bor} and suppose temporarily that $\bf G$ is $\QQ$-simple. By \cite{BB:compactification}, Theorem 3.7,
\begin{align*}
X^*=\coprod_{\bf P}e({P}),
\end{align*} 
where $\bf P$ varies over the maximal parabolic subgroups of $\bf G$ and $\bf G$ itself, and $e(P)$ is the unique rational boundary component normalized by $P$. 

Let $\bf P$ be a maximal parabolic subgroup of $\bf G$ and let ${\bf P}_0$ be a minimal parabolic subgroup of $\bf G$ contained in $\bf P$. Let $\bf A$ be a maximal split subtorus of $\bf G$ contained in ${\bf P}_0$. Therefore, $\bf P$ is equal to ${\bf P}_J$ for a unique subset $J\subseteq\Delta:=\Delta({\bf P}_0,{\bf A})$. 

Fix the {\it canonical numbering}
\begin{align*}
\Delta=\{\alpha_1,\cdots,\alpha_s\},
\end{align*}
as in \cite{BB:compactification}, Section 2.8, and, for $b=0,1,\ldots,s-1$, let
\begin{align*}
\theta(b):=\{\alpha_{b+1},\cdots,\alpha_s\}.
\end{align*}
For $b=s$, let $\theta(b)$ denote the empty set. By \cite{BJ:compactifications}, III.4.2, if $\alpha_b\in\Delta$ is the unique element not in $J$, we may identify $e(P)$ with the boundary symmetric space $X_Q$, where ${\bf Q}:={\bf P}_{\theta(b)}$. As in, \cite{BJ:compactifications}, Proposition I.11.3, this does not depend on the choice of ${\bf P}_0$ or $\bf A$.

Now let $\bf P$ be any parabolic subgroup of $\bf G$, and choose ${\bf P}_0$, $\bf A$ and $\Delta$ in the manner above. That is, ${\bf P}={\bf P}_J$ for a unique subset $J\subseteq\Delta$ (though, we have chosen a new $\Delta$, of course). We let $b=0,1,\ldots,s$ be the smallest index such that $I:=\theta(b)$ is contained in $J$, and we let $I'$ denote $J\setminus I$. 

If $b=0$, then ${\bf P}$ is equal to ${\bf G}$. Therefore, assume that $b>0$. As in \cite{BJ:compactifications}, Proposition I.11.3,
\begin{align*}
X_P=X_{P_I}\times X_{P_{I'}}
\end{align*}
and this splitting is independent of our choices for ${\bf P}_0$ and $\bf A$. Note that $X_{P_I}$ is also the boundary component $e(P_{\alpha_b})$ of $X^*$ corresponding to the maximal parabolic subgroup ${\bf P}_{\alpha_b}$. In particular, varying over the proper parabolic subgroups $\bf P$ of $\bf G$, the projection maps
\begin{align*}
X_P\rightarrow X_{P_I}
\end{align*}
extend the identity map on $X$ to a surjective $\Gamma$-equivariant map from 
\begin{align*}
_{\QQ}\overline{X}^S_{\rm max}\rightarrow X^*
\end{align*}
that is also continuous by Proposition \ref{map} below. More generally, when $\bf G$ is semisimple of adjoint type, $X$ is equal to a product $X_1\times\cdots\times X_r$ of irreducible factors corresponding to the $\QQ$-simple factors of $\bf G$. The partial compactifications $_{\QQ}\overline{X}^S_{\rm max}$ and $X^*$ of $X$ are then the products of the partial compactifications of the irreducible factors and we obtain a map between them by taking the product of the maps defined above.

\begin{prop}\label{map}
The surjective $\Gamma$-equivariant map from $_{\QQ}\overline{X}^S_{\rm max}$ to $X^*$ defined above is continuous. Therefore, we obtain a continuous surjective map
\begin{align*}
\overline{\Gamma\backslash X}^S_{\rm max}\rightarrow\overline{\Gamma\backslash X}^{BB}.
\end{align*}
\end{prop}

\begin{proof}
See \cite{BJ:compactifications}, Proposition III.15.2 and Proposition III.15.4.
\end{proof}

\subsection{Conjecture for Baily-Borel compactification}

Consider the situation described in Section \ref{bor}. If ${\bf G}={\bf G}_1\times\cdots\times{\bf G}_r$ denotes the decomposition of $\bf G$ into its $\QQ$-simple factors, we say that a parabolic subgroup $\bf P$ of $\bf G$ is of type $BB$ if it is equal to a product of parabolic subgroups ${\bf P}_i$ of ${\bf G}_i$ such that ${\bf P}_i$ is either maximal or ${\bf G}_i$ itself.

As before, if $\bf H$ is a connected algebraic subgroup of $\bf G$ of type $\mathcal{H}$ and $g\in G$, the homogeneous probability measure on $\Gamma\backslash G$ associated with ${\bf H}$ and $g$ pushes forward to $\overline{\Gamma\backslash X}^{BB}$ under the natural maps
\begin{align*}
\Gamma\backslash G\rightarrow\Gamma\backslash X\rightarrow\overline{\Gamma\backslash X}^{BB}.
\end{align*} 
We refer to this probability measure as the {\it homogeneous probability measure on $\overline{\Gamma\backslash X}^{BB}$ associated with $\bf H$ and $g$}. Similarly, if $\bf P$ is a parabolic subgroup of $\bf G$ of type $BB$, $\bf H$ is a subgroup of ${\bf P}$ of type $\mathcal{H}$ and $g\in P$, we can define the {\it homogeneous probability measure on $\overline{\Gamma\backslash X}^{BB}$ associated with $\bf P$, $\bf H$ and $g$} in precisely the same way via the natural maps
\begin{align*}
\Gamma_P\backslash P\rightarrow\Gamma_{P_I}\backslash X_{P_I}\rightarrow\overline{\Gamma\backslash X}^{BB}
\end{align*}
for some set $I$ as constructed in Section \ref{rel}.

The following conjecture is a more precise version of Conjecture \ref{introconj} in this setting.

\begin{conj}\label{bbconj}
For each $n\in\NN$, let ${\bf H}_n$ be a connected algebraic subgroup of $\bf G$ of type $\mathcal{H}$, let $g_n\in G$ and let $\mu_n$ be the homogeneous probability measure on $\overline{\Gamma\backslash X}^{BB}$ associated with ${\bf H}_n$ and $g_n$. Suppose that $(\mu_n)_{n\in\NN}$ converges to a limit $\mu\in\mathcal{P}(\overline{\Gamma\backslash X}^{BB})$. Then $\mu$ is homogeneous.

Furthermore, if $\mu$ is associated with a parabolic subgroup ${\bf P}$ of $\bf G$ of type $BB$, a connected algebraic subgroup $\bf H$ of ${\bf P}$ of type $\mathcal{H}$ and an element $g\in P$, then ${\bf H}_n$ is contained in $\bf H$ for $n$ large enough.
\end{conj}

Finally, we show that Conjecture \ref{mainconj} implies Conjecture \ref{bbconj}.

\begin{teo}\label{maintobb}
Consider the situation described in Conjecture \ref{bbconj}. If the conclusion of Conjecture \ref{mainconj} holds (that is, for the homogeneous probability measures on $\overline{\Gamma\backslash X}^S_{\rm max}$ associated with the ${\bf H}_n$ and $g_n$), then the conclusion of Conjecture \ref{bbconj} holds. 
\end{teo}

\begin{proof}
Let $\sigma_n$ denote the homogeneous probability measure on $\overline{\Gamma\backslash X}^S_{\rm max}$ associated with ${\bf H}_n$ and $g_n$. By Proposition \ref{map}, there exists a continuous surjective map 
\begin{align*}
\pi:\overline{\Gamma\backslash X}^S_{\rm max}\rightarrow\overline{\Gamma\backslash X}^{BB},
\end{align*}
which is the quotient of the map $_{\QQ}\overline{X}^S_{\rm max}\rightarrow X^*$ described in Section \ref{rel}. Since $\pi$ is the identity map on $\Gamma\backslash X$, the homogeneous probability measure $\mu_n$ on $\overline{\Gamma\backslash X}^{BB}$ associated with ${\bf H}_n$ and $g_n$ is equal to $\pi_*\sigma_n$.

Therefore, suppose that there exists a parabolic subgroup ${\bf P}$ of $\bf G$, a connected algebraic subgroup $\bf H$ of ${\bf P}$ of type $\mathcal{H}$ and an element $g\in P$ such that some subsequence of $(\sigma_n)_{n\in\NN}$ converges to the homogeneous probability measure $\sigma$ on $\overline{\Gamma\backslash X}^S_{\rm max}$ associated with $\bf P$, $\bf H$ and $g$. Extract such a subsequence and suppose that ${\bf H}_n$ is contained in $\bf H$ for $n$ large enough. We conclude that $(\mu_n)_{n\in\NN}$ converges to $\pi_*\sigma$. 

As in Section \ref{rel}, for each $i\in\{1,\ldots r\}$, there exists a maximal split torus ${\bf A}_i$ of ${\bf G}_i$, a set $\Delta_i$ of $\QQ$-simple roots of ${\bf G}_i$ with respect to ${\bf A}_i$ with the canonical numbering
\begin{align*}
\Delta_i=\{\alpha_{i,1},\ldots,\alpha_{i,s_i}\},
\end{align*} 
and subsets $I_i=\{\alpha_{i,b_i},\ldots,\alpha_{i,s_i}\}\subseteq J_i\subseteq\Delta_i$ such that ${\bf P}$ is the product of the ${\bf P}_{J_i}$ and the map $_{\QQ}\overline{X}^S_{\rm max}\rightarrow X^*$ is the product of the natural projections
\begin{align*}
X_{P_{J_i}}\rightarrow X_{P_{I_i}}.
\end{align*}

If $I_i$ is not equal to $\Delta_i$, we let ${\bf P}_i:={\bf P}_{\alpha_{i,b_i}}$. Otherwise, we let ${\bf P}_i:={\bf G}_i$. We let $\bf Q$ denote the product of the ${\bf P}_i$, which is a parabolic subgroup of $\bf G$ of type $BB$. Then, as a boundary component in $X^*$, the product of the $X_{P_{I_i}}$ is equal to $e(Q)$. Since $\bf P$ is contained in $\bf Q$, we see that $\pi_*\sigma$ is the homogeneous probability measure on $\overline{\Gamma\backslash X}^{BB}$ associated with $\bf P$, $\bf H$ and $g$. The result follows. 
\end{proof}

\begin{rem}
If $\overline{\Gamma\backslash X}^S_{\tau}$ is another (well-defined) Satake compactification of $\Gamma\backslash X$, then, by \cite{BJ:compactifications}, Proposition III.15.2, there is a continuous surjection
\begin{align*}
\overline{\Gamma\backslash X}^S_{\rm max}\rightarrow\overline{\Gamma\backslash X}^S_{\tau}
\end{align*} 
and the proof of Theorem \ref{maintobb} generalises to $\overline{\Gamma\backslash X}^S_{\tau}$. We direct the reader to \cite{BJ:compactifications}, I.4.39 for the construction.
\end{rem}

\section{The criterion for convergence in $\Gamma\backslash G$}\label{notationP}\label{s4}

\subsection{The $d_{\bf P,K}$ functions}\label{dfunctions} 

Let ${\bf G}$ be a reductive algebraic group and let $K$ be a maximal compact subgroup of $G$. Let $\bf P$ be a proper parabolic subgroup of $\bf G$ and let $n_{\bf P}$ denote the dimension of $\mathfrak{n}_{\bf P}$. Consider the $n_{\bf P}$-th exterior product
\begin{align*}
V_{\bf P}:=\wedge^{n_{\bf P}}\mathfrak{g}
\end{align*}
of $\mathfrak{g}$ and let $L_{\bf P}$ denote the one-dimensional subspace given by $\wedge^{n_{\bf P}}\mathfrak{n}_{\bf P}$. Then $V_{\bf P}$ is a linear representation of $\bf G$ and, since $\bf P$ normalizes $\bf N_P$, $L_{\bf P}$ is a linear representation of $\bf P$. That is, $\bf P$ acts on $L_{\bf P}$ via a character $\chi_{\bf P}$. Clearly, 
\begin{align}\label{chiP}
\chi_{{\bf P}_{|{\bf A_P}}}=\sum_{\alpha\in\Phi(P,A_{\bf P})}n_{\alpha}\alpha,
\end{align}
where $n_{\alpha}$ is the dimension of the corresponding root space in $\mathfrak{g}$.

\begin{lem}\label{fundweights}
Let $\Delta$ be a set of simple $\QQ$-roots of $\bf G$ with respect to $\bf A$. For each $\alpha\in\Delta$, let $\chi_\alpha$ denote the restriction of $\chi_{{\bf P}_\alpha}$ to $\bf A$. Then the $\chi_\alpha$ constitute a set of quasi-fundamental weights in $X^*({\bf A})_\QQ$.
\end{lem}

\begin{proof}
Let $\alpha\in\Delta$. By Lemma \ref{orth}, we obtain a decomposition
\begin{align*}
X^*({\bf A})_\QQ=X^*({\bf A}^\alpha)_\QQ\oplus X^*({\bf A}_\alpha)_\QQ
\end{align*}
that is orthogonal with respect to $(\cdot,\cdot)$. Therefore,
\begin{align*}
(\chi_\alpha,\alpha)=(\pi_2(\chi_\alpha),\pi_2(\alpha))=(d_\alpha\pi_2(\alpha),\pi_2(\alpha))=d_\alpha\cdot(\pi_2(\alpha),\pi_2(\alpha)),
\end{align*}
for some $d_\alpha\in\QQ_{>0}$, where the first equality follows from the fact that $\pi_1(\chi_\alpha)=0$ (since ${\bf A}^\alpha$ is contained in the kernel of the character $\chi_\alpha$), and the second equality follows from the fact that $\chi_\alpha$ is a sum of positive roots and $X^*({\bf A}_\alpha)_\QQ$ is one-dimensional. On the other hand, if $\beta\in I$, then
\begin{align*}
(\chi_\alpha,\beta)=(\pi_2(\chi_\alpha),\pi_1(\beta))=0,
\end{align*}
where we use the fact that $\pi_2(\beta)=0$.
\end{proof}

Fix a $K$-invariant norm $\|\cdot\|_{\bf P}$ on $V_{\bf P}\otimes_{\QQ}\RR$ and let $v_{\bf P}\in L_{\bf P}\otimes_{\QQ}\RR$ be such that $\|v_{\bf P}\|_{\bf P}=1$. We obtain a function $d_{{\bf P},K}$ on $G$ defined by
\begin{align*}
d_{{\bf P},K}(g):=\|g\cdot v_{\bf p}\|_{\bf P}.
\end{align*}
Note that, for any $g\in G$, we can write $g=kp$, where $k\in K$ and $p\in P$. Therefore,
\begin{align*}
d_{{\bf P},K}(g)=\|g\cdot v_{\bf p}\|_{\bf P}=\|p\cdot v_{\bf p}\|_{\bf P}=\chi_{\bf P}(p)\cdot\|v_{\bf p}\|_{\bf P}=\chi_{\bf P}(p)
\end{align*}
(note that $\chi_{\bf P}$ is necesarily positive on the connected component $P$).
In particular, $d_{{\bf P},K}$ is a {\it function on $G$ of type $(P,\chi_{\bf P})$}, as defined in \cite{borel:arithmetic-groups}, Section 14.1. Furthermore, it does not depend on the choices of $\|\cdot\|_{\bf P}$ and $v_{\bf P}$. The following lemma will allow us to relate the behaviour of the $\alpha\in\Phi(P,A_{\bf P})$ with the behaviour of $d_{{\bf P},K}$.

\begin{lem}\label{dalpha}
Let $g\in G$. Then
\begin{align*}
d_{{\bf P},K}(g^{-1})=\prod_{\alpha\in\Phi(P,A_{\bf P})}\alpha(a_{{\bf P},K}(g))^{-n_\alpha}.
\end{align*}
\end{lem}

\begin{proof}
First we decompose
\begin{align*}
g=nmak\in N_{\bf P}M_{\bf P}A_{\bf P}K,
\end{align*}
where, by definition, $a=a_{{\bf P},K}(g)$. Therefore, since $d_{{\bf P},K}$ is left $K$-invariant and trivial on $H_{\bf P}$ (because $\bf H_P$ has no rational characters),
\begin{align*}
d_{{\bf P},K}(g^{-1})=d_{{\bf P},K}(a^{-1})=\chi_{\bf P}(a)^{-1}.
\end{align*}
Therefore, the result follows from (\ref{chiP}).
\end{proof}

Now let $\Gamma$ be an arithmetic subgroup of $G_\QQ$. The following property of $d_{{\bf P},K}$ was observed in \cite{DM:uniflows}, Lemma 2.4.

\begin{lem}\label{bounded}
Let $f\in{\bf G}(\QQ)$ and let $g\in G$. Then there exists a constant $c>0$ such that $d_{{\bf P},K}(g\gamma f)\geq c$ for all $\gamma\in\Gamma$.
\end{lem}

\begin{proof}
Let $\delta>0$. Then, by \cite{DM:uniflows}, Lemma 2.4, the set
\begin{align*}
\Gamma_\delta:=\{\gamma\in\Gamma:d_{{\bf P},K}(g\gamma f)<\delta\}
\end{align*}
is finite. If $\Gamma_\delta$ is empty, we are done. Otherwise, let 
\begin{align*}
c:=\min\{d_{{\bf P},K}(g\gamma f):\gamma\in\Gamma_\delta\}.
\end{align*}	
Since $c>0$, the proof is complete.
\end{proof}

\subsection{The criterion}\label{crit}
Let $\bf G$ be a connected algebraic group with no rational characters and let $\bf L$ be a Levi subgroup of $\bf G$. Then $\bf G$ is the semidirect product of $\bf L$ and ${\bf N}:={\bf N}_{\bf G}$. We denote by $\pi$ the natural (surjective) morphism from ${\bf G}$ to ${\bf L}$.

Let ${\bf P}_0$ be a minimal parabolic subgroup of $\bf L$ and let $\bf A$ be a maximal split subtorus of $\bf L$ contained in ${\bf P}_0$. Let $K$ denote a maximal compact subgroup of $L$ such that $A$ is invariant under the Cartan involution of $G$ associated with $K$. For any proper parabolic subgroup $\bf P$ of $\bf L$, we obtain a function $d_{{\bf P},K}$ on $L$, as defined in Section \ref{dfunctions}, and, for each $\alpha\in\Delta:=\Delta({\bf P}_0,{\bf A})$, we write $d_\alpha:=d_{{\bf P}_\alpha,K}$.

Let $\Gamma$ be an arithmetic subgroup of $G_\QQ$ and let $\Gamma_L:=\pi(\Gamma)$. By \cite{milne:shimura}, Proposition 3.2, $\Gamma_L$ is an arithmetic subgroup of ${\bf L}(\QQ)$. By \cite{borel:arithmetic-groups}, Th\'eor\`eme 13.1, there exists a finite subset $F$ of $L_\QQ$ and a $t>0$ such that
\begin{align*}
L=KA_t\omega F^{-1}\Gamma_L,
\end{align*}
where $\omega$ is a compact subset of $H_{{\bf P}_0}$ and 
\begin{align*}
A_t:=\{a\in A:\alpha(a)\leq t\text{ for all }\alpha\in\Delta\}.
\end{align*}

\begin{defini}
We refer to a set $F$ as above as a $\Gamma$-set for $\bf L$. 
\end{defini}

Fix a $\Gamma$-set $F$ for $\bf L$ and let $\sigma$ denote the Lebesgue measure on $\RR$. We will require the following result due to Dani and Margulis. 

\begin{teo}[Cf. \cite{DM:uniflows}, Theorem 2]\label{GOh}
For any $\epsilon>0$ and $\theta>0$, there exists a compact subset $C:=C(\epsilon,\theta)$ of $\Gamma_L\backslash L$ such that, for every unipotent one-parameter subgroup $\{u(t)\}_{t\in\RR}$ of $L$ and every $l\in L$, either
\begin{align*}
\sigma(\{t\in[0,T]:\Gamma_L l^{-1}u(t)^{-1}\in C\})\geq(1-\epsilon)T
\end{align*}
for all large $T\in\RR$, or there exists $\alpha\in\Delta$ and $\lambda\in\Gamma_L F$ such that $d_\alpha(l\lambda)<\theta$ and
\begin{align*}
l^{-1}u(t)l\in\lambda P_\alpha\lambda^{-1}
\end{align*}
for all $t\in\RR$.
\end{teo}

\begin{proof}
In the case that $\bf L$ is semisimple, \cite{DM:uniflows}, Theorem 2 states that there exists a compact subset $C:=C(\epsilon,\theta)$ of $L/\Gamma_L$ such that, for every unipotent one-parameter subgroup $\{u(t)\}_{t\in\RR}$ of $L$ and every $l\in L$, either
\begin{align*}
\sigma(\{t\in[0,T]:u(t)l\Gamma_L \in C\})\geq(1-\epsilon)T
\end{align*}
for all large $T\in\RR$, or there exists $\alpha\in\Delta$ and $\lambda\in\Gamma_L F$ such that $d_\alpha(l\lambda)<\theta$ and
\begin{align*}
l^{-1}u(t)l\in\lambda P_\alpha\lambda^{-1}
\end{align*}
for all $t\in\RR$. We have a natural homeomorphism
\begin{align*}
\phi:L/\Gamma_L\rightarrow\Gamma_L\backslash L
\end{align*}
defined by $l\Gamma_L\mapsto\Gamma_Ll^{-1}$, and $u(t)l\Gamma_L \in C$ if and only if $\Gamma_L l^{-1}u(t)^{-1}\in \phi(C)$. This concludes the proof in the case that $\bf L$ is semisimple.

Now consider the semisimple group ${\bf L}^\ad$, that is, the quotient of ${\bf L}$ by its centre ${\bf Z}$, which is a torus. We let $\ad$ denote the natural surjective morphism from $\bf L$ to ${\bf L}^{\ad}$. Since every parabolic subgroup of $\bf L$ contains ${\bf Z}$, $\ad$ induces a bijection between parabolic subgroups of $\bf L$ and ${\bf L}^\ad$. In particular, ${\bf P}^\ad_0:=\ad({\bf P}_0)$ is a minimal parabolic subgroup of ${\bf L}^\ad$ and ${\bf A}^\ad:=\ad({\bf A})$ is a maximal split subtorus of ${\bf L}^\ad$ contained in ${\bf P}^\ad_0$. In particular, for each element $\alpha\in\Delta^\ad:=\Delta({\bf P}^\ad_0,{\bf A}^\ad)$, we obtain a maximal parabolic subgroup ${\bf P}_\alpha$ of ${\bf L}^\ad$ and a character $\chi_\alpha$ on ${\bf P}_\alpha$. The restriction of $\ad$ to $\bf A$ yields an embedding
\begin{align*}
\ad^*:X^*({\bf A}^\ad)\rightarrow X^*({\bf A})
\end{align*}
and, since the action of $\bf A$ on $\mathfrak{n}_{{\bf P}_0}$ (which we may identify with $\mathfrak{n}_{{\bf P}^\ad_0}$), factors through ${\bf A}^\ad$, we have $\Delta=\ad^*(\Delta^\ad)$. It follows that, for any $a\in A$ and any $\alpha\in\Delta$,
\begin{align*}
\chi_{\ad^*(\alpha)}(a)=\chi_\alpha(\ad(a)).
\end{align*}

Now let $l\in L$ and let $\ad^*(\alpha)\in\Delta$, for some $\alpha\in\Delta^\ad$. Writing
\begin{align}\label{addecomp}
l=kamn\in KA_{\ad^*(\alpha)}M_{\ad^*(\alpha)}N_{\ad^*(\alpha)}
\end{align}
we have
\begin{align}\label{ddad}
d_{\ad^*(\alpha)}(l)=\chi_{\ad^*(\alpha)}(a)=\chi_\alpha(\ad(a))=d_\alpha(\ad(l)),
\end{align}
where the last equality comes from applying $\ad$ to (\ref{addecomp}).

By \cite{milne:shimura}, Proposition 5.1, the induced maps from $L$ to $L^{\ad}$ and from $A$ to $A^\ad$ are surjective. In particular, $K^{\ad}:=\ad(K)$ is a maximal compact subgroup of $L^{\ad}$. By \cite[Proposition 3.2]{milne:shimura}, $\ad(\Gamma_L)$ is an arithmetic subgroup of $L^\ad_\QQ$. It follows that
\begin{align*}
L^\ad=\ad(L)=K^\ad A^\ad_t\ad(\omega)\ad(F)^{-1}\ad(\Gamma_L).
\end{align*}
That is, $\ad(F)$ is a $\ad(\Gamma_L)$-set for ${\bf L}^\ad$. Therefore, fix $\epsilon>0$ and $\theta>0$, and let $C^\ad$ denote the compact subset of $\ad(\Gamma_L)\backslash L^\ad$ afforded to us by Theorem \ref{GOh}. We claim that the preimage $C$ of $C^\ad$ under the natural map
\begin{align*}
\Gamma_L\backslash L\rightarrow \ad(\Gamma_L)\backslash L^\ad
\end{align*}
is compact. This is because the fibre above each point is isomorphic to $\Gamma_L\cap Z(L)\backslash Z(L)$; since $\bf G$ was assumed to have no rational characters, neither does ${\bf Z(L)}$, and so $\Gamma_L\cap Z(L)\backslash Z(L)$ is compact, by \cite{milne:shimura}, Theorem 3.3.

Let $\{u(t)\}_{t\in\RR}$ be a unipotent one-parameter subgroup of $L$. Then $\{\ad(u(t))\}_{t\in\RR}$ is a unipotent one-parameter subgroup of $L^{\ad}$. Let $l\in L$ such that 
\begin{align*}
\sigma(\{t\in[0,T]:\Gamma_L l^{-1}u(t)^{-1}\in C\})<(1-\epsilon)\cdot T
\end{align*}
for arbitrarily large $T\in\RR$. Then
\begin{align*}
\sigma(\{t\in[0,T]:\ad(\Gamma_L)\ad(l)^{-1}\ad(u(t))^{-1}\in C^{\ad}\})<(1-\epsilon)\cdot T
\end{align*}
for arbitrarily large $T\in\RR$. Therefore, by Theorem \ref{GOh}, there exists $\alpha\in\Delta^\ad$ and $\ad(\lambda)\in\ad(\Gamma_L)\ad(F)$ such that $d_\alpha(\ad(l)\ad(\lambda))<\theta$ and
\begin{align*}
\ad(l)^{-1}\ad(u(t))\ad(l)\in\ad(\lambda)P_\alpha\ad(\lambda)^{-1}
\end{align*}
for all $t\in\RR$. Therefore, the result follows from (\ref{ddad}).
\end{proof}

For any connected algebraic subgroup $\bf H$ of $\bf G$ and any $g\in G$, we define
\begin{align*}
\delta_{K,\Delta,F}({\bf H},g):={\rm inf}\{d_\alpha(\pi(g)^{-1}\lambda):\lambda\in\Gamma_LF,\ \alpha\in\Delta,\ {\bf H}\subset{\bf N}\lambda{\bf P}_{\alpha}\lambda^{-1}\}
\end{align*}
(where we take the value to be $\infty$ if the infimum is varying over the empty set).
By Lemma \ref{bounded}, we have $\delta_{K,\Delta,F}({\bf H},g)>0$. Our criterion is the following.

\begin{teo}\label{criterionprop}
For each $n\in\NN$, let ${\bf H}_n$ be a connected algebraic subgroup of $\bf G$ of type $\mathcal{H}$, let $g_n\in G$ and let $\mu_n$ be the homogeneous probability measure on $\Gamma\backslash G$ associated with ${\bf H}_n$ and $g_n$. Assume that
\begin{align*}
\liminf_{n\rightarrow\infty}\ \delta_{K,\Delta,F}({\bf H}_n,g_n)>0.
\end{align*}
Then the set $\{\mu_n\}_{n\in\NN}$ is sequentially compact in $\mathcal{P}(\Gamma\backslash G)$. 
\end{teo}

\begin{proof}
By Prokhorov's Theorem, it suffices to show that the set of measures $\{\mu_n\}_{n\in\NN}$ is tight on $\Gamma\backslash G$. That is, for every $\epsilon>0$, there exists a compact subset $C$ of $\Gamma\backslash G$ such that
\begin{align*}
\mu_n(C)\geq 1-\epsilon,\text{ for all }n\in\NN.
\end{align*} 
By abuse of notation, denote also by $\pi$ the natural surjection from $\Gamma\backslash G$ to $\Gamma_L\backslash L$ and suppose that the set $\{\pi_*(\mu_n)\}_{n\in\NN}$ of pushforward mesures were tight on $\Gamma_L\backslash L$. By definition, for any $\epsilon>0$, there exists a compact set $C_L$ of $\Gamma_L\backslash L$ such that 
\begin{align*}
\pi_*(\mu_n)(C_L)\geq 1-\epsilon,\text{ for all }n\in\NN.
\end{align*}
Since arithmetic quotients of unipotent algebraic groups are compact, it follows, as before, that $\pi^{-1}(C_L)$ is compact. Since
\begin{align*}
\mu_n(\pi^{-1}(C_L))=\pi_*(\mu_n)(C_L)\geq 1-\epsilon,\text{ for all }n\in\NN,
\end{align*}
we conclude that $\{\mu_n\}_{n\in\NN}$ is tight on $\Gamma\backslash G$. 

Therefore, it suffices to show that $\{\pi_*(\mu_n)\}_{n\in\NN}$ is tight on $\Gamma_L\backslash L$. To that end, fix an $\epsilon>0$. 

For each $n\in\NN$, we let ${\bf L}_n$ and $l_n$ denote $\pi({\bf H}_n)$ and $\pi(g_n)$, respectively, noting that ${\bf L}_n$ is a connected algebraic subgroup of $\bf L$ of type $\mathcal{H}$. We let $\{u_n(t)\}$ be a unipotent one-parameter subgroup of $l^{-1}_nL_nl_n$ such that the trajectory $\{\Gamma_L l_nu^{-1}_n(t)\}_{t\in\RR}$ is uniformly distributed in $\Gamma_L\backslash\Gamma_L L_nl_n$. That is, for any bounded continuous function $f$ on $\Gamma_L\backslash L$, we have
\begin{align}\label{birkoff}
\lim_{T\rightarrow\infty}\ \frac{1}{T}\int^T_0f(\Gamma_L  l_nu^{-1}_n(t))\ dt=\int_{\Gamma_L\backslash L} f\ d\pi_*(\mu_n).
\end{align}
(The existence of such a subgroup is guaranteed by Birkhoff's Ergodic Theorem.) For any $\theta>0$, we are afforded, by Theorem \ref{GOh}, a compact set $C_\theta:=C(\epsilon/2,\theta)$ of $\Gamma_L\backslash L$ such that, for each $n\in\NN$, either
\begin{align}\label{alt1}
\sigma(\{t\in[0,T]:\Gamma_L l_nu^{-1}_n(t)\in C_\theta\})\geq(1-\epsilon/2)T,
\end{align}
for all large $T\in\RR$, or there exists $\alpha\in\Delta$ and $\lambda\in\Gamma_L F$ such that $d_\alpha(l^{-1}_n\lambda)<\theta$ and
\begin{align*}
l_nu_n(t)l^{-1}_n\in\lambda P_\alpha\lambda^{-1}
\end{align*}
for all $t\in\RR$.

If (\ref{alt1}) holds, for $n\in\NN$, it follows immediately from (\ref{birkoff}) that
\begin{align*}
\pi_*(\mu_n)(C_\theta)\geq 1-\epsilon.
\end{align*}
Otherwise, we conclude that there exist $\alpha_n\in\Delta$ and $\lambda_n\in\Gamma_L F$ such that $d_{\alpha_n}(l^{-1}_n\lambda_n)<\theta$ and
\begin{align*}
l_nu_n(t)l^{-1}_n\in\lambda_n P_{\alpha_n}\lambda^{-1}_n
\end{align*}
for all $t\in\RR$. Since ${\bf P}_{\alpha_n}$ is defined over $\QQ$, the subspace $\Gamma_L\backslash\Gamma_L \lambda^{-1}_nP_{\alpha_n}\lambda_n$ is closed in $\Gamma_L\backslash L$ and, since $\Gamma_L\backslash\{\Gamma_Ll_nu^{-1}_n(t)l^{-1}_n\}_{t\in\RR}$ is dense in $\Gamma_L\backslash\Gamma_L L_n$, we conclude that $\Gamma\backslash\Gamma L_n$ is contained in $\Gamma_L\backslash\Gamma_L\lambda^{-1}_nP_{\alpha_n}\lambda_n$. This implies that the Lie algebra of $L_n$ is contained in the Lie algebra of $\lambda^{-1}_n{\bf P}_{\alpha_n}\lambda_n$ and, therefore, ${\bf L}_n$ is contained in $\lambda^{-1}_n{\bf P}_{\alpha_n}\lambda_n$ itself. In particular, ${\bf H}_n$ is contained in ${\bf N}\lambda_n {\bf P}_{\alpha_n}\lambda^{-1}_n$ and
\begin{align*}
\delta_{K,\Delta,F}({\bf H}_n,g_n)<\theta. 
\end{align*}

Therefore, for each $k\in\NN$, let $\theta_k>0$ such that $\theta_k\rightarrow 0$ as $k\rightarrow\infty$. Let $C_k:=C_{\theta_k}$. As explained above, either
\begin{align*}
\pi_*(\mu_n)(C_k)\geq 1-\epsilon,
\end{align*}
for all $n\in\NN$, and we conclude that $\{\pi_*(\mu_n)\}_{n\in\NN}$ is tight on $\Gamma_L\backslash L$, or there exists $n_k\in\NN$ such that
\begin{align*}
\delta_{K,\Delta,F}({\bf H}_{n_k},g_{n_k})<\theta_k. 
\end{align*}
However, the latter contradicts the assumption of the theorem, hence, the proof is complete.
\end{proof}

\section{The criterion for convergence in $\overline{\Gamma\backslash X}^S_{\rm max}$}\label{s5}

Let $\bf G$ denote a semisimple algebraic group of adjoint type and let $K$ denote a maximal compact subgroup of $G$. Let $X$ denote the symmetric space $G/K$ and let $\Gamma$ denote an arithmetic subgroup of $G_\QQ$. Our criterion is as follows.

\begin{teo}\label{convergencecriterion}
For each $n\in\NN$, let ${\bf H}_n$ denote a connected algebraic subgroup of $\bf G$ of type $\mathcal{H}$, let $g_n$ denote an element of $G$ and let $\mu_n$ denote the homogeneous probability measure on $\overline{\Gamma\backslash X}^S_{\rm max}$ associated with ${\bf H}_n$ and $g_n$.

Suppose that there exists a parabolic subgroup ${\bf P}$ of $\bf G$ such that,
\begin{enumerate}
\item[(i)] for all $n\in\NN$, ${\bf H}_n$ is contained in $\bf P$,
\item[(ii)] we can write
\begin{align*}
g_n=p_na_nk_n\in PA_{\bf P}K,
\end{align*}	
such that
\begin{align*}
\alpha(a_n)\rightarrow\infty,\text{ as }n\rightarrow\infty,\text{ for all }\alpha\in\Phi(P,A_{\bf P}),
\end{align*}
and, 
\item[(iii)] if we denote by $\nu_n$ the homogeneous probability measure on $\Gamma_P\backslash P$ associated with ${\bf H}_n$ and $p_n$, then $(\nu_n)_{n\in\NN}$ converges to $\nu\in\mathcal{P}(\Gamma_P\backslash P)$.
\end{enumerate}

Then there exists a connected algebraic subgroup $\bf H$ of ${\bf P}$ of type $\mathcal{H}$ and an element $g\in P$ such that $(\mu_n)_{n\in\NN}$ converges to the homogeneous probability measure on $\overline{\Gamma\backslash X}^S_{\rm max}$ associated with $\bf P$, $\bf H$ and $g$, and, furthermore, ${\bf H}_n$ is contained in ${\bf H}$ for $n$ large enough.
\end{teo}

\begin{proof}
Consider the natural maps
\begin{align*}
\pi:\Gamma\backslash G\rightarrow\overline{\Gamma\backslash X}^S_{\rm max}\ \text{  and }\ \pi_{\bf P}:\Gamma_{P}\backslash P\rightarrow\Gamma_P\backslash X_P\rightarrow\overline{\Gamma\backslash X}^S_{\rm max}
\end{align*} 
as in Section \ref{main}. The measure $\mu_n$ is equal to $\pi_*(\iota_*(\nu_n)\cdot a_n)$, where $\iota$ denotes the natural inclusion
of $\Gamma_P\backslash P$ in $\Gamma\backslash G$.
Furthermore, by Theorem \ref{ratner}, there exists a connected algebraic subgroup $\bf H$ of ${\bf P}$ of type $\mathcal{H}$ and an element $g\in P$ such that $\nu$ is the homogeneous probability measure on $\Gamma_P\backslash P$ associated with ${\bf H}$ and $g$, and ${\bf H}_n$ is contained in $\bf H$ for $n$ large enough. Therefore, it suffices to show that the sequence $(\mu_n)_{n\in\NN}$ converges to the pushforward $\pi_{{\bf P}*}(\nu)$.

To that end, let $f$ be a continuous function on $\overline{\Gamma\backslash X}^S_{\rm max}$ (which is automatically bounded because $\overline{\Gamma\backslash X}^S_{\rm max}$ is compact). Fix an $\epsilon>0$ and let $C$ be a compact subset of $\Gamma_P\backslash P$ such that $\nu_n(C)>1-\epsilon$, for all $n\in\NN$, and $\nu(C)>1-\epsilon$ as well.
Finally, let $\phi$ and $\psi$ be two bounded non-negative continuous functions on $\Gamma_P\backslash P$ such that $\phi+\psi=1$, the set ${\rm supp}(\phi)$ is compact, and the set ${\rm supp}(\psi)$ is contained in the complement of $C$.

We are interested in
\begin{align*}
\int_{\overline{\Gamma\backslash X}^S_{\rm max}}f\ d\pi_*((\iota_*\nu_n)\cdot a_n)&=\int_{\Gamma\backslash\Gamma Pa_n} f\circ\pi\ d\iota_*(\nu_n)\cdot a_n\\
=\int_{\Gamma\backslash\Gamma P} f\circ\pi\circ r_{a_n}\ d\iota_*(\nu_n)&=\int_{\Gamma_P\backslash P} f\circ\pi\circ r_{a_n}\circ\iota\ d\nu_n,
\end{align*}
where $r_{a_n}$ denotes the homeomorphism of $\Gamma\backslash G$ given by multiplication by $a_n$ on the right.
We write the last integral as the sum
\begin{align*}
\int_{\Gamma_P\backslash P} (f\circ\pi\circ r_{a_n}\circ\iota)\phi\ d\nu_n+\int_{\Gamma_P\backslash P} (f\circ\pi\circ r_{a_n}\circ\iota)\psi\ d\nu_n,
\end{align*}
which, by assumption, is equal to
\begin{align*}
\int_{\Gamma_P\backslash P} (f\circ\pi\circ r_{a_n}\circ\iota)\phi\ d\nu_n+O(\epsilon).
\end{align*}

It follows immediately from \cite{BJ:compactifications}, III.11.2 that, for any $p\in P$,
\begin{align*}
\pi(r_{a_n}(\iota(\Gamma_Pp)))\rightarrow\pi_{\bf P}(\Gamma_Pp)
\end{align*} 
uniformly on compact sets, as $n\rightarrow\infty$.
Therefore, as functions on $\Gamma_P\backslash P$, we see that
\begin{align*}
f\circ\pi\circ r_{a_n}\circ\iota \rightarrow f\circ\pi_{\bf P}
\end{align*}
uniformly, as $n\rightarrow\infty$, for all $\Gamma_Pp\in C$. In particular,
\begin{align*}
\int_{\Gamma_P\backslash P} (f\circ\pi\circ r_{a_n}\circ\iota)\phi\ d\nu_n=\int_{\Gamma_P\backslash P} (f\circ\pi_{\bf P})\phi\ d\nu_n+O(\epsilon),
\end{align*}
for sufficiently large $n$.

We have
\begin{align*}
\int_{\Gamma_P\backslash P} (f\circ\pi_{\bf P})\phi\ d\nu_n&=\int_{\Gamma_P\backslash P} (f\circ\pi_{\bf P})\phi\ d\nu+O(\epsilon),
\end{align*}
for sufficiently large $n$, whereas,
\begin{align*}
\int_{\Gamma_P\backslash P} (f\circ\pi_{\bf P})\phi\ d\nu&=\int_{\Gamma_P\backslash P} (f\circ\pi_{\bf P})\ d\nu+O(\epsilon),
\end{align*}
by the definition of $C$. Therefore, since
\begin{align*}
\int_{\Gamma_P\backslash P} (f\circ\pi_{\bf P})\ d\nu&=\int_{\overline{\Gamma\backslash X}^S_{\rm max}} f\ d\pi_{{\bf P}*}(\nu),
\end{align*}
the result follows from the fact that $\epsilon>0$ can be chosen arbitrarily.
\end{proof}

\section{Groups of $\QQ$-rank $0$}\label{s6}

Having established our criteria for convergence, we now move on to proving various cases of Conjecture \ref{mainconj}. 

\begin{teo}\label{rank0}
Conjecture \ref{mainconj} holds when $\bf G$ has $\QQ$-rank $0$ (that is, $\bf G$ is $\QQ$-anisotropic).
\end{teo}

\begin{proof}
In this case, $\bf G$ has only one rational parabolic subgroup, namely, $\bf G$ itself. Therefore, by Theorem \ref{criterionprop}, after possibly extracting a subsequence, the sequence of homogeneous probability measures on $\Gamma\backslash G$ associated with the ${\bf H}_n$ and the $g_n$ is convergent in $\mathcal{P}(\Gamma\backslash G)$. By Theorem \ref{ratner}, the limit measure is the homogeneous probability measure on $\Gamma\backslash G$ associated with a connected algebraic subgroup $\bf H$ of $\bf G$ of type $\mathcal{H}$ and an element $g\in G$, and, furthermore, ${\bf H}_n$ is contained in ${\bf H}$ for all $n$ large enough. We pushforward all measures to $\Gamma\backslash X$ and the theorem follows.
\end{proof}

\section{Groups of $\QQ$-rank $1$}\label{s7}

\begin{teo}\label{rank1}
Conjecture \ref{mainconj} holds when $\bf G$ has $\QQ$-rank $1$.
\end{teo}

\begin{proof}
Consider the situation described in Section \ref{main} and suppose that $\bf G$ has $\QQ$-rank $1$. Let $\bf A$ be a maximal split subtorus of $\bf G$ (so the dimension of $\bf A$ is equal to $1$) and let ${\bf P}$ denote a minimal parabolic subgroup of $\bf G$ containing $\bf A$. The set $\Delta:=\Delta({\bf P},{\bf A})$ contains only one element, which we denote $\alpha$. After possibly replacing $K$, we can and do assume that $A$ is invariant under the Cartan involution of $G$ associated with $K$. 

Let $F$ denote a $\Gamma$-set for $\bf G$. Then, in the notation of Section \ref{crit}, we obtain, for each $n\in\NN$, a positive real number $\delta_{K,\Delta,F}({\bf H}_n,g_n)$. Suppose that
\begin{align}\label{criteriononlevel1}
\liminf_{n\rightarrow\infty}\ \delta_{K,\Delta,F}({\bf H}_n,g_n)>0.
\end{align}
By Theorem \ref{criterionprop}, after possibly extracting a subsequence, the sequence of homogeneous probability measures on $\Gamma\backslash G$ associated with the ${\bf H}_n$ and the $g_n$ is convergent in $\mathcal{P}(\Gamma\backslash G)$, in which case the proof concludes as in the proof of Theorem \ref{rank0}.

Therefore, suppose that (\ref{criteriononlevel1}) does not hold. Since $\bf G$ has $\QQ$-rank $1$, every proper parabolic subgroup of $\bf G$ is minimal. Since the minimal parabolic subgroups of $\bf G$ belong to a single ${\bf G}(\QQ)$-conjugacy class, it follows from \cite{borel:arithmetic-groups}, Proposition 15.6 that every maximal parabolic subgroup of $\bf G$ is conjugate to ${\bf P}$ by an element of $\Gamma F$. Therefore, by Lemma \ref{dalpha}, we can and do extract a subsequence such that, for every $n\in\NN$, there exists $\lambda_n\in\Gamma F$ such that ${\bf H}_n$ is contained in $\lambda_n{\bf P}\lambda^{-1}_n$ and, if 
\begin{align*}
\lambda^{-1}_ng_n=h_na_nk_n\in H_{\bf P}A_{\bf P}K,
\end{align*}
then, 
\begin{align*}
\alpha(a_n)\rightarrow\infty,\text{ as }n\rightarrow\infty.
\end{align*}
Furthermore, after possibly extracting a subsequence, we can and do assume that $\lambda_n=\gamma_n c$, where $\gamma_n\in\Gamma$, and $c\in F$ is fixed. Therefore, we can and do replace ${\bf H}_n$ with $\lambda^{-1}_n{\bf H}_n\lambda_n$ and $g_n$ with $\lambda^{-1}_ng_n$ and relabel them ${\bf H}_n$ and $g_n$, respectively. In particular, ${\bf H}_n$ is contained in ${\bf P}$
and
\begin{align*}
g_n=h_na_nk_n\in H_PA_PK.
\end{align*}
Since ${\bf M}_{\bf P}$ is $\QQ$-anisotropic, it contains no proper parabolics, and we conclude as in the proof of Theorem \ref{rank0} that, after possibly extracting a subsequence, the sequence of homogeneous probability measures on $\Gamma_P\backslash P$ associated with the ${\bf H}_n$ and the $h_n$ is convergent in $\mathcal{P}(\Gamma_P\backslash P)$. Therefore, the theorem follows from Theorem \ref{convergencecriterion}.
\end{proof}

We remark that, when ${\bf G}$ has $\QQ$-rank $1$, every proper parabolic subgroup $\bf P$ is minimal (and maximal). In particular, ${\bf H_P}$ is anisotropic over $\QQ$. It can happen, then, that $M_P$ is compact for all such $\bf P$ and that the boundary components of ${\Gamma\backslash X}^S_{\rm max}$ (except for $\Gamma\backslash X$) are points. In which case, Theorem \ref{rank1} generalises \cite{MS:ergodic}, Corollary 1.3, from a one-point-compactification to a finitely-many-points-compactification.

Unfortunately, it does not seem possible to generalise the above approach to the general case by way of induction. The argument breaks down at the second stage as one moves to a non-maximal parabolic subgroup. This is in some sense because of Lemma \ref{warning}. We proceed to the case ${\bf G}={\bf SL}_3$, in part to explain this problem more explicitly.

\section{The case of ${\bf SL}_3$}\label{s8}

Unable to go beyond the rank $1$ case in full generality, we consider the specific case of ${\bf G}={\bf SL}_3$. Already in this case, one can appreciate the complexity of the general problem, and the obstructions to performing an inductive argument.

\begin{teo}\label{teo8.1}
Conjecture \ref{mainconj} holds when ${\bf G}={\bf SL}_3$, $K={\rm SO}(3)$, and $\Gamma={\bf SL}_3(\ZZ)$.
\end{teo}

\begin{proof}

Let $\Delta:=\{\alpha_1,\alpha_2\}$ be the set of simple $\QQ$-roots of $\bf G$ with respect to the maximal diagonal torus $\bf A$, where $\alpha_1$ is defined by
\begin{align*}
{\rm diag}(x,y,(xy)^{-1})\mapsto xy^{-1},
\end{align*}
and $\alpha_2$ is defined by
\begin{align*}
{\rm diag}(x,y,(xy)^{-1})\mapsto y(xy)=xy^2.
\end{align*}
By \cite{borel:arithmetic-groups}, Section 1.10, it suffices in this case to take $F:=\{1\}$.

Suppose that
\begin{align}\label{criteriononG}
\liminf_{n\rightarrow\infty}\ \delta_{K,\Delta,F}({\bf H}_n,g_n)>0.
\end{align}
Then, by Theorem \ref{criterionprop}, the set of measures corresponding to our sequence is sequentially compact in $\mathcal{P}(\Gamma\backslash G)$ and our claim follows with ${\bf P}={\bf G}$ from Theorem \ref{ratner}.

Therefore, suppose that (\ref{criteriononG}) does not hold. By Lemma \ref{dalpha}, we can and do extract a subsequence such that, for some $\alpha\in\Delta$, and for every $n\in\NN$, there exists $\gamma_n\in\Gamma$ such that ${\bf H}_n$ is contained in $\gamma_n{\bf P}_\alpha\gamma^{-1}_n$ and, if we write
\begin{align*}
\gamma^{-1}_ng_n=h_na_nk_n\in H_{\alpha}A_{\alpha}K,
\end{align*}
then
\begin{align*}
\alpha(a_n)\rightarrow\infty,\text{ as }n\rightarrow\infty.
\end{align*}
Therefore, we replace ${\bf H}_n$ with $\gamma^{-1}_n{\bf H}_n\gamma_n$ and $g_n$ with $\gamma^{-1}_ng_n$ and relabel them ${\bf H}_n$ and $g_n$, respectively. That is, ${\bf H}_n$ is contained in ${\bf H}_\alpha$
and
\begin{align*}
g_n=h_na_nk_n\in H_{\alpha}A_{\alpha}K.
\end{align*}
By the symmetry of our arguments, we can and do assume that $\alpha=\alpha_2$. That is,
\begin{align*}
\alpha_2(a_n)\rightarrow\infty,\text{ as }n\rightarrow\infty.
\end{align*}

The elements of ${\bf M}_\alpha\cong\SL_2$ have the form
\begin{align*}
\left(\begin{array}{ccc}
* & * & 0 \\
* & * & 0 \\
0 & 0 & 1
\end{array}\right)
\end{align*}
and the restriction of $\beta:=\alpha_1$ to ${\bf M}_{\alpha}$ yields a set of $\QQ$-simple roots with respect to the maximal torus ${\bf A}^{\alpha}$ of ${\bf M}_\alpha$. We let $K_{\alpha}:=M_{\alpha}\cap K\cong {\rm SO}(2)$ and, again, we can and do choose $F_{\alpha}:=\{1\}$. If
\begin{align}\label{criteriononMalpha}
\liminf_{n\rightarrow\infty}\ \delta_{K_{\alpha},\{\beta\},F_{\alpha}}({\bf H}_n,h_n)>0,
\end{align}
then, by Theorem \ref{criterionprop}, the conditions of Theorem \ref{convergencecriterion} are satisfied with ${\bf P}={\bf P}_\alpha$, and the result follows.

Therefore, suppose that (\ref{criteriononMalpha}) does not hold. We can and do extract a subsequence such that, for every $n\in\NN$, there exists $\gamma_n\in{\bf M}_{\alpha}(\ZZ)$ such that ${\bf H}_n$ is contained in the parabolic subgroup ${\bf N}_{\alpha}\gamma_n{\bf P}^{\alpha}_\emptyset\gamma^{-1}_n$ of ${\bf H}_\alpha$, where ${\bf P}^\alpha_\emptyset$ denotes the standard minimal parabolic subgroup of ${\bf M}_{\alpha}$ whose elements are of the form
\begin{align*}
\left(\begin{array}{ccc}
* & * & 0 \\
0 & * & 0 \\
0 & 0 & 1
\end{array}\right),
\end{align*}
and, if we write
\begin{align*}
\gamma^{-1}_nh_n=s_nb_nl_n\in N_{\alpha}H^{\alpha}_\emptyset\cdot A^{\alpha}\cdot K_{\alpha},
\end{align*}
where ${\bf H}^\alpha_\emptyset:={\bf H}_{{\bf P}^\alpha_{\emptyset}}$, then
\begin{align*}
\beta(b_n)\rightarrow\infty,\text{ as }n\rightarrow\infty.
\end{align*}
Therefore, we replace ${\bf H}_n$ with $\gamma^{-1}_n{\bf H}_n\gamma_n$ and $g_n$ with $\gamma^{-1}_ng_n$ and relabel them ${\bf H}_n$ and $g_n$, respectively. That is, ${\bf H}_n$ is contained in ${\bf N}_\emptyset={\bf N}_\alpha{\bf H}^\alpha_\emptyset$
and
\begin{align*}
g_n=s_n\cdot b_na_n\cdot l_nk_n\in H_\emptyset\cdot A_\emptyset\cdot K,
\end{align*}
where
\begin{align*}
\beta(b_n)\rightarrow\infty\text{ and }\alpha(a_n)\rightarrow\infty,\text{ as }n\rightarrow\infty.
\end{align*}
However (and herein lies the problem), whereas
\begin{align*}
\beta(b_na_n)\rightarrow\infty,\text{ as }n\rightarrow\infty
\end{align*}
(because $\beta(a_n)=1$, for all $n\in\NN$), the behaviour of $\alpha(b_na_n)$ is not clear, because
\begin{align*}
\alpha(b_n)=\beta(b_n)^{-1/2}\rightarrow 0,\text{ as }n\rightarrow\infty.
\end{align*}
Recall that, by Lemma \ref{warning}, the exponent here is necessarily non-positive. Of course, if 
\begin{align*}
\alpha(b_na_n)\rightarrow\infty,\text{ as }n\rightarrow\infty,
\end{align*}
then the result follows from Theorem \ref{convergencecriterion}, with ${\bf P}={\bf P}_\emptyset$, where we use the fact that, since ${\bf N}_\emptyset$ is unipotent, the space ${\bf N}_\emptyset(\ZZ)\backslash N_\emptyset$ is compact. Therefore, we can and do suppose that
\begin{align*}
\alpha(b_na_n)\rightarrow c\in[0,\infty),\text{ as }n\rightarrow\infty.
\end{align*}
It is sufficient, as we do, to restrict to the following cases.

{\bf \underline{Case 1}:} Suppose that, for all $n\in\NN$, ${\bf H}_n$ is not contained in ${\bf N}_{\beta}$. 

We know that ${\bf H}_n$ is contained in 
\begin{align*}
{\bf H}_{\beta}={\bf N}_{\beta}{\bf M}_{\beta},
\end{align*} 
for all $n\in\NN$, and so, by assumption, the projection of $H_n$ to $M_{\beta}$ is $N^{\beta}_\emptyset:=N_{{\bf P}^\beta_\emptyset}$, where ${\bf P}^\beta_\emptyset$ denotes the standard minimal parabolic of ${\bf M}_{\beta}$, whose elements are of the form
\begin{align*}
\left(\begin{array}{ccc}
1 & 0 & 0 \\
0 & * & * \\
0 & 0 & *
\end{array}\right).
\end{align*}
The Bruhat decomposition of ${\bf M}_{\beta}$ (see \cite{borel:arithmetic-groups}, 11.4 (ii)) yields the decomposition
\begin{align*}
{\bf M}_{\beta}(\QQ)={\bf P}^\beta_\emptyset(\QQ)\cup{\bf P}^\beta_\emptyset(\QQ)\eta{\bf P}^\beta_\emptyset(\QQ),
\end{align*}
where the union is disjoint and 
\begin{align*}
\eta:=\left(\begin{array}{ccc}
1 & 0 & 0 \\
0 & 0 & -1 \\
0 & 1 & 0
\end{array}\right)
\end{align*} 
the non-trivial element of the Weyl group.
Therefore, since, by Corollary \ref{norm}, the normaliser of ${\bf N}^\beta_\emptyset$ in ${\bf M}_{\beta}$ is ${\bf P}^\beta_\emptyset$ and, by Lemma \ref{int}, the normaliser of ${\bf P}^\beta_\emptyset$ is ${\bf P}^\beta_\emptyset$ itself, we deduce that the projection of ${\bf H}_n$ to ${\bf M}_{\beta}$ is only contained in one parabolic subgroup of ${\bf M}_{\beta}$ (namely, ${\bf P}^\beta_\emptyset$).

We can write
\begin{align*}
b_na_n=\beta_n\alpha_n\in A^{\beta}A_{\beta}.
\end{align*}
In particular,
\begin{align*}
\alpha(\beta_n)=\alpha(\beta_n\alpha_n)=\alpha(b_na_n)\rightarrow c,\text{ as }n\rightarrow\infty,
\end{align*}
where we use the fact that $\alpha_n\in A_\beta$.
Therefore, if we put $K_\beta:=M_\beta\cap K$ and $F_{\beta}:=\{1\}$, then, by the previous discussion,
\begin{align*}
\liminf_{n\rightarrow\infty}\ \delta_{K_{\beta},\{\alpha\},F_{\beta}}({\bf H}_n,s_n\beta_n)=\lim_{n\rightarrow\infty}\alpha(\beta_n)^{-1}=c^{-1}>0.
\end{align*}
Hence, by Theorem \ref{criterionprop},
the set of probability measures on ${\bf H}_{\beta}(\ZZ)\backslash H_{\beta}$ associated with the ${\bf H}_n$ and $s_n\beta_n$ is sequentially compact. Furthermore, we claim that
\begin{align*}
\beta(\alpha_n)\rightarrow\infty,\text{ as }n\rightarrow\infty.
\end{align*}
To see this, write
\begin{align*}
b_n={\rm diag}(y_n,y^{-1}_n,1)\text{ and }a_n={\rm diag}(x_n,x_n,x^{-2}_n),
\end{align*}
so that 
\begin{align*}
y_n\rightarrow\infty\text{ and }x_n\rightarrow\infty,\text{ as }n\rightarrow\infty.
\end{align*}
Therefore, if
\begin{align*}
\beta_n={\rm diag}(1,w_n,w^{-1}_n)\text{ and }\alpha_n={\rm diag}(z^{-2}_n,z_n,z_n),
\end{align*}
we have $z^{-2}_n=y_nx_n$, which yields
\begin{align}\label{zed}
\beta(\alpha_n)=z^{-3}_n=(y_nx_n)^{3/2}\rightarrow\infty,\text{ as }n\rightarrow\infty.
\end{align}
Therefore, after possibly extracting a subsequence, the result follows in this case from Theorem \ref{convergencecriterion}, with ${\bf P}={\bf P}_{\beta}$.

{\bf \underline{Case 2}:} Suppose that, for all $n\in\NN$, ${\bf H}_n$ is contained in ${\bf N}_{\beta}$.

{\bf\underline{Case 2.1}:} Suppose that $c\in(0,\infty)$.

After possibly extracting a subsequence, the sequence of measures on ${\bf H}_{\beta}(\ZZ)\backslash H_{\beta}$ associated with the ${\bf H}_n$ and $s_n\beta_n$ converges and the result follows in this case, with ${\bf P}={\bf P}_{\beta}$, from Theorem \ref{convergencecriterion}, and the fact demonstrated above that
\begin{align*}
\beta(\alpha_n)\rightarrow\infty,\text{ as }n\rightarrow\infty.
\end{align*}

{\bf \underline{Case 2.2}:} Suppose that $c=0$.

{\bf\underline{Case 2.2.1}:} Suppose that, for all $n\in\NN$, $s_n\in N_{\beta}$.

For all $n\in\NN$, 
\begin{align*}
\eta{\bf H}_n\eta^{-1}\subseteq{\bf N}_{\beta}\subsetneq {\bf N}_{\emptyset}={\bf H}_{\emptyset}
\end{align*}
and
\begin{align*}
\eta s_n\eta^{-1}\in N_{\beta}\subsetneq N_{\emptyset}=H_{\emptyset}.
\end{align*}
Furthermore,
\begin{align*}
\alpha(\eta \beta_n\alpha_n\eta^{-1})=\alpha(\eta \beta_n\eta^{-1})\alpha(\alpha_n)=\alpha(\beta_n)^{-1}\rightarrow\infty,\text{ as }n\rightarrow\infty,
\end{align*}
and
\begin{align*}
\beta(\eta \beta_n\alpha_n\eta^{-1})=\eta(\beta)(\beta_n\alpha_n)=(\beta+\alpha)(\beta_n\alpha_n)=y_nx_n^3\rightarrow\infty,\text{ as }n\rightarrow\infty.
\end{align*}
Since $\eta\in\Gamma\cap K$, the homogeneous probability measure on $\overline{\Gamma\backslash X}^S_{\rm max}$ associated with ${\bf H}_n$ and $g_n$ is equal to the homogeneous probability measure on $\overline{\Gamma\backslash X}^S_{\rm max}$ associated with $\eta{\bf H}_n\eta^{-1}$ and $\eta g_n\eta^{-1}$. Therefore, the result follows in this case, with ${\bf P}={\bf P}_\emptyset$, from Theorem \ref{convergencecriterion}.

{\bf\underline{Case 2.2.2}:} Suppose that, for all $n\in\NN$, $s_n\notin N_{\beta}$. 

By assumption,
\begin{align*}
s_n=\left(\begin{array}{ccc}
1 & * & * \\
0 & 1 & t_n \\
0 & 0 & 1
\end{array}\right),
\end{align*}
where $t_n\in\RR\setminus\{0\}$, and so
\begin{align*}
s_n\beta_n=\left(\begin{array}{ccc}
1 & * & * \\
0 & w_n & t_nw^{-1}_n \\
0 & 0 & w^{-1}_n
\end{array}\right).
\end{align*}
Therefore, by \cite{borel:arithmetic-groups}, Section 1.10, there exist $\gamma_n\in{\bf M}_\beta(\ZZ)$ and $m_n\in K_\beta$ such that 
\begin{align*}
\gamma_ns_n\beta_nm_n=\left(\begin{array}{ccc}
1 & * & * \\
0 & 1 & u_n \\
0 & 0 & 1
\end{array}\right)\left(\begin{array}{ccc}
1 & * & * \\
0 & v_n & 0 \\
0 & 0 & v^{-1}_n
\end{array}\right),
\end{align*}
where $u_n\in[0,1]$ and $v_n\geq t:=2/\sqrt{3}$. 

{\bf\underline{Case 2.2.2.1}:} Suppose that $v_n$ remains bounded, as $n\rightarrow\infty$. 

In this case, we can rewrite
\begin{align*}
\Gamma H_ng_n=\Gamma\cdot \gamma_nH_n\gamma^{-1}_n\cdot \gamma_ns_n\beta_n m_n\cdot \alpha_n \cdot m_n^{-1}l_nk_n
\end{align*}
and the result follows, with ${\bf P}={\bf P}_\beta$, from Theorem \ref{convergencecriterion} since the set of measures on ${\bf H}_{\beta}(\ZZ)\backslash H_{\beta}$ associated with the $\gamma_n{\bf H}_n\gamma^{-1}_n$ and $\gamma_ns_n\beta_nm_n$ is sequentially compact and
\begin{align*}
\beta(\alpha_n)\rightarrow\infty,\text{ as }n\rightarrow\infty.
\end{align*}

{\bf\underline{Case 2.2.2.2}:} Suppose that $v_n\rightarrow\infty$, as $n\rightarrow\infty$. 

If we denote
\begin{align*}
c_n:=\left(\begin{array}{ccc}
1 & 0 & 0 \\
0 & v_n & 0 \\
0 & 0 & v^{-1}_n
\end{array}\right),
\end{align*}
then 
\begin{align*}
\alpha(c_n\alpha_n)=\alpha(c_n)\rightarrow\infty,\text{ as }n\rightarrow\infty
\end{align*}
and we can write
\begin{align*}
\Gamma H_ng_n=\Gamma H_n\cdot \nu_n \cdot c_n\alpha_n \cdot m_n^{-1}l_nk_n,
\end{align*}
where $\nu_n\in N_\emptyset$ and so the result depends on the behaviour of $\beta(c_n\alpha_n)$.

{\bf\underline{Case 2.2.2.2.1}:} Suppose that
\begin{align*}
\beta(c_n\alpha_n)\rightarrow\infty,\text{ as }n\rightarrow\infty.
\end{align*}

In this case, the result follows, with ${\bf P}={\bf P}_\emptyset$, from Theorem \ref{convergencecriterion}. 

{\bf\underline{Case 2.2.2.2.2}:} Suppose that $\beta(c_n\alpha_n)$ converges to a limit $d\in(0,\infty)$. 

We can write
\begin{align*}
c_n\alpha_n=d_ne_n\in A_\alpha A^\alpha,
\end{align*}
and so
\begin{align*}
\beta(e_n)=\beta(d_ne_n)=\beta(c_n\alpha_n)\rightarrow d,\text{ as }n\rightarrow\infty.
\end{align*}
Therefore, after possibly extracting a subsequence, the result follows, with ${\bf P}={\bf P}_\alpha$, from Theorem \ref{convergencecriterion}, since
\begin{align*}
\Gamma H_ng_n=\Gamma H_n\cdot \nu_ne_n\cdot d_n \cdot m_n^{-1}l_nk_n
\end{align*}
and the the set of measures on ${\bf H}_{\alpha}(\ZZ)\backslash H_{\alpha}$ associated with the ${\bf H}_n$ and $\nu_ne_n$ is sequentially compact, whereas, since $d_n$ is a bounded distance from $c_n\alpha_n$,
\begin{align*}
\alpha(d_n)\rightarrow\infty,\text{ as }n\rightarrow\infty.
\end{align*} 

{\bf\underline{Case 2.2.2.2.3}:} Suppose that
\begin{align*}
\beta(c_n\alpha_n)\rightarrow 0,\text{ as }n\rightarrow\infty.
\end{align*}

{\bf\underline{Case 2.2.2.2.3.1}:} Suppose that, for all $n\in\NN$, ${\bf H}_n$ is not contained in the unipotent group whose elements are of the form
\begin{align*}
\left(\begin{array}{ccc}
1 & 0 & * \\
0 & 1 & 0 \\
0 & 0 & 1
\end{array}\right).
\end{align*}
  
That is, as before, the projection of the ${\bf H}_n$ to ${\bf M}_{\alpha}$ is contained in only one parabolic subgroup of ${\bf M}_{\alpha}$, namely, ${\bf P}^\beta_\emptyset$, the standard minimal parabolic, whose elements are of the form
\begin{align*}
\left(\begin{array}{ccc}
* & * & 0 \\
0 & * & 0 \\
0 & 0 & 1
\end{array}\right).
\end{align*}
Therefore, precisely as in Case 1, we deduce that the set of measures on ${\bf H}_{\alpha}(\ZZ)\backslash H_{\alpha}$ associated with the ${\bf H}_n$ and $\nu_ne_n$ is sequentially compact, and the result follows, with ${\bf P}={\bf P}_\alpha$, from Theorem \ref{convergencecriterion} since
\begin{align*}
\alpha(d_n)\rightarrow\infty,\text{ as }n\rightarrow\infty.
\end{align*}

{\bf\underline{Case 2.2.2.2.3.2}:} Suppose that, for all $n\in\NN$, every element of ${\bf H}_n$ is of the form
\begin{align*}
\left(\begin{array}{ccc}
1 & 0 & * \\
0 & 1 & 0 \\
0 & 0 & 1
\end{array}\right).
\end{align*}
In particular, ${\bf H}:={\bf H}_n$ is fixed, and we write $H_0$ for the fundamnetal domain of $H$ whose elements are of the form
\begin{align*}
\left(\begin{array}{ccc}
1 & 0 & u \\
0 & 1 & 0 \\
0 & 0 & 1
\end{array}\right),\text{ where }u\in[0,1].
\end{align*}

Since $H$ is contained in the centre of $N_\emptyset$, we have
\begin{align*}
\Gamma Hg_n=\Gamma H_0 g_n=\Gamma \nu_n(c_n\alpha_n)\cdot(c_n\alpha_n)^{-1}H_0(c_n\alpha_n)\cdot m_n^{-1}l_nk_n.
\end{align*}
Also, after possibly extracting a subsequence, we can and do assume that the sequence of points $\Gamma\nu_nc_n\alpha_nK$ converges to a point 
\begin{align*}
\Gamma_Px_P\in\Gamma_P\backslash X_P\subseteq\overline{\Gamma\backslash X}^S_{\rm max},
\end{align*}
where $\bf P$ is a parabolic subgroup of $\bf G$. We claim, then, that the sequence of measures on $\overline{\Gamma\backslash X}^S_{\rm max}$ associated with $\bf H$ and the $g_n$ converges to the Dirac measure associated with $\Gamma_Px_P$. 

To see this, note that, by \cite{BJ:compactifications}, Theorem III.11.9, there exist $\gamma_n\in\Gamma$ such that
\begin{align*}
\gamma_n\nu_nc_n\alpha_n=(\pi_n,\rho_n,o_n,\mu_n\kappa_n)\in N_{\bf P}\times A_{\bf P}\times\exp\mathfrak{a}^\perp_P\times M_PK
\end{align*}
(where we are using the horospherical decomposition of $G$ with respect to $\bf P$ as in \cite{BJ:compactifications}, III.11.2), where $\mu_nK_P\rightarrow x_P\in X_P$, and 
\begin{align*}
\alpha(\rho_n)\rightarrow\infty,\text{ as }n\rightarrow\infty,\text{ for all }\alpha\in\Phi(P,A_{\bf P}).
\end{align*}
On the other hand, since, by (\ref{zed}),
\begin{align*}
(\beta+\alpha)(c_n\alpha_n)=v_nz^{-3}_n\rightarrow\infty,\text{ as }n\rightarrow\infty,
\end{align*}
every sequence $(\theta_n)_{n\in\NN}$, with $\theta_n\in(c_n\alpha_n)^{-1}H_0(c_n\alpha_n)$ converges to the identity. In particular, the sequence $(\kappa_n\theta_n\kappa_n^{-1})_{n\in\NN}$ also converges to the identity so, if we write
\begin{align*}
\kappa_n\theta_n\kappa_n^{-1}=(\pi'_n,\rho'_n,o'_n,\mu'_n\kappa'_n),
\end{align*}
then the individual components must each converge to the identity. Therefore,
\begin{align*}
\gamma_n\nu_nc_n\alpha_n\theta_nK=\pi_n\rho_no_n\mu_n\kappa_n\theta_n\kappa^{-1}_nK=\pi_n\rho_no_n\mu_n\pi'_n\rho'_no'_n\mu'_nK,
\end{align*}
which we can write as
\begin{align*}
(\pi''_n,\rho_n\rho'_n,o_no'_n,\mu_n\mu'_nK),
\end{align*}
where $\mu_n\mu'_nK\rightarrow x_P\in X_P$, and
\begin{align*}
\alpha(\rho_n\rho'_n)\rightarrow\infty,\text{ as }n\rightarrow\infty,\text{ for all }\alpha\in\Phi(P,A_{\bf P}).
\end{align*}
In particular, by \cite{BJ:compactifications}, III.11.2, any sequence $(\Gamma x_n)_{n\in\NN}$, with $x_n\in Hg_nK$, converges to $\Gamma_Px_P$, from which the claim follows.
\end{proof}

\section{A product of modular curves}\label{s9}

Next, we digress to prove a far simpler case, but one for which there is a Baily-Borel compactification. 

First, we prove a simple lemma. 

\begin{lem}
Let $\bf H$ denote a connected algebraic subgroup of ${\bf SL}_2$ of type $\mathcal{H}$. Then $\bf H$ is either trivial, ${\bf SL}_2$ itself, or $\gamma{\bf N}\gamma^{-1}$, for some $\gamma\in{\bf SL}_2(\ZZ)$, where ${\bf N}$ denotes the unipotent radical of the Borel subgroup $\bf B$ of ${\bf SL}_2$ consisting of upper triangular matrices.
\end{lem}

\begin{proof}
If $\bf H$ is semisimple, then ${\bf H}={\bf SL}_2$. If $\bf H$ is non-trivial
and not semisimple, then $\bf N_H$ is non-trivial. Furthermore, since $\bf N_H$ is unipotent, it is contained
in some parabolic subgroup of ${\bf SL}_2$. Hence, $\bf N_H$ is the unipotent radical of a Borel subgroup. Since every Borel subgroup of ${\bf SL}_2$ is of the form $\gamma {\bf B}\gamma^{-1}$, for some $\gamma\in{\bf SL}_2(\ZZ)$, we deduce that ${\bf N_H}=\gamma{\bf N}\gamma^{-1}$, for some $\gamma\in{\bf SL}_2(\ZZ)$. Furthermore, by definition, ${\bf H}$ is contained in the normaliser of $\bf N_H$, which is $\gamma{\bf B}\gamma^{-1}$. That is, $\bf H$ is contained in $\gamma{\bf B}\gamma^{-1}={\bf N_H}\gamma{\bf D}\gamma^{-1}$, where $\bf D$ denotes the diagonal torus. Since $\bf H$ is of type $\mathcal{H}$, we conclude that ${\bf H=N_H}=\gamma{\bf N}\gamma^{-1}$.
\end{proof}

\begin{teo}\label{teo9.1}
Let $r\in\NN$. Conjecture \ref{mainconj} holds when ${\bf G}={\bf SL}^r_2$, $K={\rm SO}(2)^r$, and $\Gamma={\bf SL}_2(\ZZ)^r$.
\end{teo}

\begin{proof}
For each $i\in I:=\{1,\ldots,r\}$, let $\pi_i:{\bf G}\rightarrow{\bf SL}_2$ denote the projection of $\bf G$ on to its $i^{\rm th}$ factor. After possibly extracting a subsequence, we obtain a partition $I^s\cup I^u\cup I^t$ of $I$ into three disjoint subsets, where
\begin{itemize}
\item $i\in I^s$ if and only if, for all $n\in\NN$, $\pi_i({\bf H}_n)={\bf SL}_2$,
\item $i\in I^u$ if and only if, for all $n\in\NN$, $\pi_i({\bf H}_n)=\gamma_{i,n} {\bf N}\gamma^{-1}_{i,n}$, for some $\gamma_{i,n}\in {\bf SL}_2(\ZZ)$, and
\item $i\in I^t$ if and only if, for all $n\in\NN$, $\pi_i({\bf H}_n)=\{1\}.$
\end{itemize}
For each $n\in\NN$, we let $\gamma_n\in{\SL}_2(\ZZ)^r$ be the element whose $i^{\rm th}$ entry is equal to $\gamma_{i,n}$ if $i\in I^u$ and is trivial otherwise. We replace ${\bf H}_n$ with $\gamma^{-1}_n {\bf H}_n\gamma_n$ and $g_n$ with $\gamma^{-1}_n g_n$ and relabel them ${\bf H}_n$ and $g_n$, respectively. That is, the assertions above on the $\pi_i({\bf H}_n)$ now hold with $\gamma_{i,n}=1$ for all $i\in I^u$ and all $n\in\NN$.

Let $\bf D$ denote the diagonal torus of ${\bf SL}_2$. Then ${\bf A}:={\bf D}^r$ is a maximal split torus of ${\bf G}$ such that $A$ is invariant under the Cartan involution of $G$ associated with $K$. Furthermore, ${\bf P}_0:={\bf B}^r$ is a minimal parabolic subgroup of $\bf G$ (indeed, it is a Borel subgroup) that contains $\bf A$.

We let $\Delta:=\Delta({\bf P}_0,{\bf A})$, which can naturally be indexed by $I$. In particular, for each $i\in I$, we obtain a maximal standard parabolic ${\bf P}_i$ of $\bf G$ by replacing the $i^{\rm th}$ factor of $\bf G$ with a copy of $\bf B$. This exhausts the maximal standard parabolic subgroups. 

For each $i\in I$, we obtain a function $d_i:=d_{{\bf P}_i,K}$ on $G$. As before, $F:=\{1\}$ is a $\Gamma$-set for $\bf G$. Note that the maximal parabolic subgroups of $\bf G$ that contain the ${\bf H}_n$ are those ${\bf P}_i$ for $i\in I^u$ and those 
$\gamma {\bf P}_i\gamma^{-1}$ for $i\in I^t$ and any $\gamma\in\Gamma$. 

After possibly extracting a subsequence, we can define a partition $I^t_+\cup I^t_{-}$ of $I^t$ into two disjoint subsets, where
\begin{itemize}
\item $i\in I^t_+$ if and only if 
\begin{align*}
\inf\{d_i(g^{-1}_n\gamma):\gamma\in{\bf SL}_2(\ZZ)^r\}\rightarrow 0,\text{ as }n\rightarrow\infty,\text{ and}
\end{align*}
\item  $i\in I^t_{-}$ if and only if 
\begin{align*}
\inf\{d_i(g^{-1}_n\gamma):\gamma\in{\bf SL}_2(\ZZ)^r\}\rightarrow c_i\in(0,\infty],\text{ as }n\rightarrow\infty.
\end{align*}
\end{itemize}
Similarly, we can define a partition $I^u_+\cup I^u_-$ of $I^u$ into two disjoint subsets, where
\begin{itemize}
\item $i\in I^u_+$ if and only if $d_i(g^{-1}_n)\rightarrow 0,\text{ as }n\rightarrow\infty$, and
\item $i\in I^u_{-}$ if and only if $d_i(g^{-1}_n)\rightarrow c_i\in(0,\infty],\text{ as }n\rightarrow\infty$.
\end{itemize}
In particular, by Lemma \ref{dalpha}, after possibly replacing ${\bf H}_n$ with $\gamma^{-1}_n{\bf H}_n\gamma_n$ and $g_n$ with $\gamma^{-1}_ng_n$, for some $\gamma_n\in\Gamma$ (with trivial entries outside of the factors in $I^t_+$), and relabelling them ${\bf H}_n$ and $g_n$, respectively, we can write $g_n=(g_{i,n})_{i=1}^r$, where $g_{i,n}\in{\bf SL}_2(\RR)$, such that
\begin{align*}
 g_{i,n}= u_{i,n}a_{i,n}k_{i,n}\in N\cdot D\cdot{\rm SO}(2)
\end{align*}
and
\begin{align*}
\phi(a_{i,n})\rightarrow \infty,\text{ as }n\rightarrow\infty,
\end{align*}
for all $i\in I^t_+\cup I^u_+$, where we denote by $\phi$ the single element of $\Delta({\bf B},{\bf D})$. We can also, without loss of generality, assume that $k_{i,n}=1$, for all $i\in I$ and for all $n\in\NN$. 

We define $h_n=(h_{i,n})_{i=1}^r\in G$ and $\theta_n=(\theta_{i,n})_{i=1}^r\in A$, where
\begin{itemize}
\item $h_{i,n}=g_{i,n}$ and $\theta_{i,n}=1$, for all $i\in I^s\cup  I^u_{-}\cup I^t_{-}$, and
\item $h_{i,n}=u_{i,n}$ and $\theta_{i,n}= a_{i,n}$, for all $i\in I^t_+\cup I^u_+$.
\end{itemize}
That is, for all $i\in I$ and all $n\in \NN$, we have $g_n=h_n\theta_n$. 

Let $J:=I^s\cup  I^u_{-}\cup I^t_{-}$ and consider the standard parabolic subgroup ${\bf P}:={\bf P}_J$ of $\bf G$. By definition,
\begin{align*}
{\bf P}=\prod_{i=1}^r{\bf P}_{J,i},
\end{align*}
where ${\bf P}_{J,i}={\bf B}$ if $i\notin J$ and ${\bf P}_{J,i}={\bf SL}_2$ otherwise. The rational Langlands decomposition with respect to $K$ is $P=N_JA_JM_J$, where
\begin{align*}
N_J=\prod_{i=1}^r N_{J,i},\ A_J=\prod_{i=1}^r A_{J,i},\text{ and }M_J=\prod_{i=1}^r M_{J,i},
\end{align*}
where
\begin{itemize}
\item for $i\notin J$, we have $N_{J,i}=N$, $A_{J,i}=D$ and $M_{J,i}=\{1\}$, and,
\item for $i\in J$, we have $N_{J,i}=\{1\}$, $A_{J,i}=\{1\}$ and $M_{J,i}={\rm SL}_2$.
\end{itemize}
Let ${\bf H}:={\bf H_P}$. Then ${\bf H}_n$ is contained in ${\bf H}$ and $h_n\in H$ for all $n\in \NN$. Let $\Gamma_H:=\Gamma\cap H$. Then, by Theorem \ref{criterionprop}, the set of homogeneous probability measures on $\Gamma_H\backslash H$ associated with the ${\bf H}_n$ and $h_n$ is sequentially compact in $\mathcal{P}(\Gamma_H\backslash H)$. Furthermore, $g_n=h_n\theta_n$, and
\begin{align*}
\alpha(\theta_n)\rightarrow\infty,\text{ as }n\rightarrow\infty,\text{ for all }\alpha\in\Phi(P,A_{\bf P}).
\end{align*}
Therefore, the result follows from Theorem \ref{convergencecriterion}.
\end{proof}

\section{Translates of the Levi of a maximal parabolic subgroup}\label{s12}

We now move on to proving cases of Conjecture \ref{mainconj} in which we impose conditions on the ${\bf H}_n$ and $g_n$ instead of the group $\bf G$.

The title of this section is slightly inaccurate; a Levi subgroup cannot be of type $\mathcal{H}$. Recall that, for any reductive algebraic group $\bf M$, we can write $\bf M$ as the almost direct product ${\bf R}_{\bf M}{\bf M}^\der$, where ${\bf M}^\der$ is the derived subgroup of $\bf M$. Then ${\bf M}^\der$ is a semisimple group and, as such, is equal to the almost direct product of its almost $\QQ$-simple factors. We write ${\bf M}^{\rm nc}$ (respectively, ${\bf M}^{\rm c}$) for the product of those factors whose underlying real Lie groups are non-compact (respectively, compact). In particular, ${\bf M}^{\rm nc}$ is of type $\mathcal{H}$. 

\begin{teo}\label{levi}
Let ${\bf P}_0$ be a minimal parabolic subgroup of $\bf G$ and let $\bf A$ be a maximal split subtorus of $\bf G$ contained in ${\bf P}_0$. Then Conjecture \ref{mainconj} holds when, for all $n\in N$, ${\bf H}_n={\bf M}^{\rm nc}_\alpha$, for some $\alpha\in\Delta:=\Delta({\bf P}_0,{\bf A})$.
\end{teo}

\begin{proof}
After possibly replacing $K$, we can and do assume that $A$ is invariant under the Cartan involution of $G$ associated with $K$. 

Let $F$ denote a $\Gamma$-set for $\bf G$. If
\begin{align}\label{maxcriteriononlevel1}
\liminf_{n\rightarrow\infty}\ \delta_{K,\Delta,F}({\bf H}_n,g_n)>0,
\end{align}
then, by Theorem \ref{criterionprop}, after possibly extracting a subsequence, the sequence of homogeneous probability measures on $\Gamma\backslash G$ associated with the ${\bf H}_n$ and the $g_n$ is convergent in $\mathcal{P}(\Gamma\backslash G)$, in which case the proof concludes as in the proof of Theorem \ref{rank0}.

Therefore, suppose that (\ref{maxcriteriononlevel1}) does not hold. By Lemma \ref{dalpha}, we can and do extract a subsequence such that, for some $\beta\in\Delta$, and for every $n\in\NN$, there exists $\lambda_n\in\Gamma F$ such that ${\bf H}_n$ is contained in $\lambda_n{\bf P}_\beta \lambda^{-1}_n$ and, if we write
\begin{align*}
\lambda^{-1}_ng_n=h_na_nk_n\in H_{\beta}A_{\beta}K,
\end{align*}
then
\begin{align*}
\beta(a_n)\rightarrow\infty,\text{ as }n\rightarrow\infty.
\end{align*}
Furthermore, after possibly extracting a subsequence, we can and do assume that $\lambda_n=\gamma_n c$, where $\gamma_n\in\Gamma$, and $c\in F$ is fixed. Therefore, we can and do replace ${\bf H}_n$ with $\lambda^{-1}_n{\bf H}_n\lambda_n$ and $g_n$ with $\lambda^{-1}_ng_n$ and relabel them ${\bf H}_n$ and $g_n$, respectively. That is, ${\bf H}_n$ is contained in ${\bf H}_\beta$
and
\begin{align*}
g_n=h_na_nk_n\in H_{\beta}A_{\beta}K.
\end{align*}

Now, if the $\QQ$-rank of $\bf G$ is $r$, then the $\QQ$-rank of ${\bf H}_n$ is $r-1$. On the other hand, if ${\bf H}_n$ were contained in a parabolic subgroup of ${\bf H}_\beta$, it would ncessarily be contained in a semisimple subgroup of $\QQ$-rank $r-2$, which is a contradiction. Therefore, $H_n$ is contained in no parabolic subgoup of ${\bf H}_\beta$, and we conclude from Theorem \ref{criterionprop}, after possibly extracting a subsequence, the sequence of homogeneous probability measures on $\Gamma_\beta\backslash P_\beta$ associated with the ${\bf H}_n$ and the $h_n$, where $\Gamma_\beta:=\Gamma\cap P_\beta$, is convergent in $\mathcal{P}(\Gamma_\beta\backslash P_\beta)$. Therefore, the result follows from Theorem \ref{convergencecriterion} with ${\bf P}={\bf P}_\beta$.
\end{proof}

\section{Translates of the unipotent radical of a minimal parabolic}\label{s10}

We will prove the following case of Conjecture \ref{mainconj}.

\begin{teo}\label{unip}
Conjecture \ref{mainconj} holds when, for each $n\in N$, ${\bf H}_n$ is equal to the unipotent radical of a minimal parabolic subgroup of $\bf G$.
\end{teo}

\begin{proof}
After possibly extracting a subsequence and conjugating, we can and do assume that, for all $n\in\NN$, ${\bf H}_n={\bf N_P}$ for a fixed minimal parabolic subgroup ${\bf P}$ of $\bf G$. Let $\bf A$ denote a maximal split subtorus of $\bf G$ contained in $\bf P$. We can and do assume that $A$ is invariant under the Cartan involution of $G$ associated with $K$. Writing
\begin{align*}
g_n=\nu_nm_na_nk_n\in N_{\bf P}M_{\bf P}A_{\bf P}K,
\end{align*}  
we see that the homogeneous probability measure on $\overline{\Gamma\backslash X}^S_{\rm max}$ associated with ${\bf N}_{\bf P}$ and $g_n$ is equal to the homogeneous probability measure associated with ${\bf N}_{\bf P}$ and $m_na_n$.

Let $\Delta:=\Delta({\bf P},{\bf A})$ and, for any subset $I\subseteq\Delta$, let ${\bf P}_I$ denote the standard parabolic subgroup of $\bf G$ associated with $I$. Then ${\bf A}={\bf A}_I{\bf A}^I$ and we write
\begin{align*}
a_n=a_{n,I}a^I_n\text{, where } a_{n,I}\in A_I\text{ and } a^I_n\in A^I.
\end{align*}

For each $\alpha\in\Delta$, let $d_{\alpha}:=d_{{\bf P}_\alpha,K}$ and let $\chi_\alpha:=\chi_{{\bf P}_\alpha}$. Let $F$ be a $\Gamma$-set for $\bf G$. Note that, for any $I\subseteq\Delta$, and any $\lambda_n\in\Gamma F\cap{\bf P}_{\alpha}(\QQ)$,
\begin{align*}
d_\alpha((a_n^{I})^{-1}m^{-1}_n\lambda_n)=\chi_\alpha((a_n^{I})^{-1})\chi_\alpha(m^{-1}_n)\chi_\alpha(\lambda_n).
\end{align*}
Furthermore, by Lemma \ref{bounded}, $\chi_\alpha(\lambda_n)=d_\alpha(\lambda_n)\geq c_1$, for some $c_1>0$ depending only on $G$ and $F$, and $\chi_\alpha(m^{-1}_n)=1$ because $m_n\in M_{\bf P}\subseteq M_\alpha$. Therefore, 
\begin{align*}
d_\alpha((a_n^{I})^{-1}m^{-1}_n\lambda_n)\geq c_1\cdot\chi_\alpha((a_n^{I})^{-1}).
\end{align*}

Choose $I\subseteq\Delta$ maximal such that there exists $c_2>0$ satisfying
\begin{align*}
\chi_\alpha((a_n^{I})^{-1})>c_2
\end{align*}
for every $n\in\NN$ and $\alpha\in\Delta$. That is, $I$ is a maximal subset such that, for the corresponding decomposition of $a_n$, all of the characters $\chi_\alpha$ are bounded below on $(a^I_n)^{-1}$. Note that such a set exists because, when $I$ is empty, ${\bf M}_I={\bf M_P}$ is anisotropic and so $a^I_n=1$.

By Lemma \ref{mustbestandard}, ${\bf N}_{\bf P}$ is contained in a parabolic subgroup $\bf Q$ of $\bf G$ if and only if ${\bf Q}={\bf P}_I$ for some subset $I\subseteq\Delta$. Therefore, Theorem \ref{criterionprop} implies that the set of homogeneous probability measures on $\Gamma_I\backslash H_I$ associated with ${\bf N}_{\bf P}$ and the $m_na^I_n$ is sequentially compact. We are using here the fact that, by Lemma \ref{int}, for any $\alpha\in\Delta$, ${\bf P}_\alpha$ is its own nomalizer in $\bf G$. Hence, for any $\lambda\in\Gamma F$, we have $\lambda{\bf P}_\alpha\lambda^{-1}={\bf P}_\alpha$ if and only if $\lambda\in\Gamma F\cap{\bf P}_{\alpha}(\QQ)$.

Therefore, Theorem \ref{unip} follows from Theorem \ref{convergencecriterion}, if we can show that
\begin{align*}
\alpha(a_{n,I})\rightarrow\infty,\text{ as }n\rightarrow\infty,\text{ for all }\alpha\in\Delta\setminus I.
\end{align*}
To that end, let $\alpha\in\Delta\setminus I$ and let $I_\alpha:=I\cup\{\alpha\}$. Recall that $I_\alpha$ restricts to a set of simple $\QQ$-roots of ${\bf M}_{I_\alpha}$ with respect to ${\bf A}^{I_\alpha}$. As in \cite{BJ:compactifications}, III.1.16, we obtain a maximal proper standard parabolic ${\bf P}^{I_\alpha}_I$ of ${\bf M}_{I_\alpha}$ and a rational Langlands decomposition
\begin{align*}
{P}^{I_\alpha}_I={N}^{I_\alpha}_I{M}^{I_\alpha}_I{A}^{I_\alpha}_I,
\end{align*}
such that 
\begin{align*}
{P}_I={N}_{I_\alpha}{N}^{I_\alpha}_I\cdot {M}^{I_\alpha}_I\cdot {A}^{I_\alpha}_I{A}_{I_\alpha}
\end{align*}
is the rational Langlands decomposition of ${P}_I$. In particular, ${\bf A}_I={\bf A}^{I_\alpha}_I{\bf A}_{I_\alpha}$ and so
\begin{align*}
a_{n,I}=a_{n,I_\alpha}a^{I_\alpha}_{n,I},\text{ where }a_{n,I_\alpha}\in A_{I_\alpha}\text{ and }a^{I_\alpha}_{n,I}\in A^{I_\alpha}_I.
\end{align*}
Therefore,
\begin{align*}
\alpha(a_{n,I})=\alpha(a_{n,I_\alpha}a^{I_\alpha}_{n,I})=\alpha(a^{I_\alpha}_{n,I}).
\end{align*}

We also have the decomposition
\begin{align*}
a_n=a_{n,I_\alpha}a^{I_\alpha}_n,\text{ where }a_{n,I_\alpha}\in A_{I_\alpha}\text{ and }a^{I_\alpha}_n\in A^{I_\alpha},
\end{align*}
and, from the maximality of $I$, we know that, after possibly extracting a subsequence, we can and do assume that, for some $\beta\in\Delta$,
\begin{align*}
\chi_{\beta}((a^{I_\alpha}_n)^{-1})\rightarrow 0,\text{ as }n\rightarrow\infty.
\end{align*}
We note that $\beta\in I_\alpha$ since, otherwise, $\chi_\beta$ would be trivial on ${\bf A}^{I_\alpha}$.
The decompositions
\begin{align*}
{\bf A}={\bf A}^{I_\alpha}{\bf A}_{I_\alpha}={\bf A}^I{\bf A}_I={\bf A}^I{\bf A}^{I_\alpha}_I{\bf A}_{I_\alpha}
\end{align*}
yield $a^{I_\alpha}_n=a^I_na^{I_\alpha}_{n,I}$ and, since $\chi_\beta((a^I_n)^{-1})>c_2$, for every $n\in\NN$, we conclude that
\begin{align*}
\chi_\beta((a^{I_\alpha}_{n,I})^{-1})\rightarrow 0\text{ as }n\rightarrow\infty.
\end{align*}

Now since $X^*({\bf A}^{I_\alpha}_I)$ is a one dimensional $\ZZ$-module and the restriction of $\alpha$ to ${\bf A}^{I_\alpha}_I$ is non-trivial, 
\begin{align*}
\chi_{\beta|{\bf A}^{I_\alpha}_I}=c_\beta(\alpha)\cdot \alpha_{|{\bf A}^{I_\alpha}_I}
\end{align*}
for some $c_\beta(\alpha)\in\QQ$, and we claim that $c_\beta(\alpha)>0$. To see this, recall from Lemma \ref{fundweights} that $\chi_\beta$ belongs to a set of quasi-fundamental weights in $X^*({\bf A})_\QQ$. Therefore, by Lemma \ref{restr}, its restriction to ${\bf A}^{I_\alpha}$ belongs to a set of quasi-fundamental weights in $X^*({\bf A}^{I_\alpha})_\QQ$. It follows from Lemma \ref{nonneg}, then, that the restriction of $\chi_\beta$ to ${\bf A}^{I_\alpha}$ is a non-negative linear combination of the elements of $I_\alpha$. Since $\chi_{\beta|{\bf A}^{I_\alpha}_I}$ is non-zero, the claim follows. We conclude that
\begin{align*}
\alpha(a_{n,I})=\alpha(a^{I_\alpha}_{n,I})\rightarrow\infty\text{ as }n\rightarrow\infty,
\end{align*}
as required.
\end{proof}

\section{Digression on Levi spheres}\label{LS}\label{s11}

Before proving a final case of Conjecture \ref{mainconj}, we recall some facts pertaining to buildings and so-called Levi spheres. We refer the reader to \cite{brown:buildings}, \cite{ronan:buildings}, and \cite{tits:buildings} for more details on buildings.

Let ${\bf G}$ be a reductive algebraic group, containing a non-trivial split torus, and let $\mathcal{B}:=\mathcal{B}({\bf G})$ be its associated (Tits) building. That is, $\mathcal{B}$ is a simplicial complex whose simplexes are in one-to-one correspondence with the (rational) parabolic subgroups of ${\bf G}$. If $s\in\mathcal{B}$ is a simplex, we denote by ${\bf P}_s$ the corresponding parabolic subgroup of ${\bf G}$, and if $\bf P$ is a parabolic subgroup of $\bf G$, we denote by $s_{\bf P}\in\mathcal{B}$ the corresponding simplex. Then $s\in\mathcal{B}$ is a face of $t\in\mathcal{B}$ if and only if ${\bf P}_t$ is contained in ${\bf P}_s$. In particular, the vertices of $\mathcal{B}$ are in one-to-one correspondence with the maximal proper parabolic subgroups of $\bf G$, and the empty simplex corresponds to $\bf G$ itself. The set of types of vertices of $\mathcal{B}$ is in bijection with the set of vertices of the (rational) Dynkin diagram of ${\bf G}$. The apartments of $\mathcal{B}$ are in one-to-one correspondence with the maximal split tori of ${\bf G}$. 

Let ${\bf P}_0$ be a minimal parabolic subgroup of ${\bf G}$ and let ${\bf A}$ be a maximal split subtorus of $\bf G$ contained in ${\bf P}_0$. The set $\Delta:=\Delta({\bf P}_0,{\bf A})$ is a set of simple $\QQ$-roots of $\bf G$ with respect to $\bf A$. Let $V:=X_*({\bf A})\otimes\RR$ and let $V^*:=X^*({\bf A})\otimes\RR$. There is a canonical perfect pairing
\begin{align*}
\langle\cdot,\cdot\rangle:V\times V^*\rightarrow\RR,
\end{align*}
and we identify $V^*$ with the dual of $V$. We choose a basis $\{\psi_\alpha\}_{\alpha\in\Delta}$ of $V^*$ such that $\langle\psi_\beta,\alpha\rangle=\delta_{\beta\alpha}$, for all $\beta,\alpha\in\Delta$.

Let $W$ denote the Weyl group of $\bf A$ (which acts linearly on $V$ and $V^*$). We equip $V$ with a $W$-invariant scalar product $(\cdot,\cdot)$, which allows us to identify $V$ with its dual and, therefore, with $V^*$. Since $(\cdot,\cdot)$ and $\langle\cdot,\cdot\rangle$ are $W$-invariant, this identification is $W$-equivariant. 

The exponential map
\begin{align*}
\Lie(A)\rightarrow A:a\mapsto\exp(a)
\end{align*}
is an isomorphism of real Lie groups, and the map
\begin{align*}
A\rightarrow V:a\mapsto (\log\alpha(a))_{\alpha\in\Delta}
\end{align*}
is also an isomorphism. Therefore, we have obtained identifications $\Lie(A)=A=V$, which are all $W$-equivariant.

For each $\alpha\in V^*$, we define a hyperplane
\begin{align*}
H_\alpha:=\{x\in V:\langle x,\alpha\rangle=0\},
\end{align*}
a half-space
\begin{align*}
\Theta_\alpha:=\{x\in V:\langle x,\alpha\rangle>0\},
\end{align*}
and its closure
\begin{align*}
\overline{\Theta}_\alpha:=\{x\in V:\langle x,\alpha\rangle\geq 0\}.
\end{align*}
Note that, for any $w\in W$, $wH_{\alpha}=H_{w\alpha}$ and $w\Theta_\alpha=\Theta_{w\alpha}$. 

For $I\subset\Delta$, we define
\begin{align*}
C_I:=\bigcap_{\alpha\in I}H_\alpha\cap\bigcap_{\alpha\notin I}\Theta_\alpha,
\end{align*}
and so
\begin{align*}
wC_I=\bigcap_{\alpha\in I}H_{w\alpha}\cap\bigcap_{\alpha\notin I}\Theta_{w\alpha},
\end{align*}
for any $w\in W$. The $wC_I$ yield a partition of $V$ and so, if we denote by $[wC_I]$ the intersection of $wC_I$ with the unit sphere $S(V)$ in $V$, we obtain a partition of $S(V)$ and, in fact, the simplices of a simplicial complex $\mathcal{S}:=\mathcal{S}(W,\Delta)$, for which the simplex $[w_1C_{I_1}]$ is a face of $[w_2C_{I_2}]$ if and only if, as subsets of $S(V)$, $[w_1C_{I_1}]$ is contained in the closure of $[w_2C_{I_2}]$. One may verify that the map
\begin{align*}
\mathcal{S}\rightarrow\mathcal{A}:[wC_I]\mapsto s_{w{\bf P}_Iw^{-1}}
\end{align*}
is a $W$-equivariant isomorphism of simplicial complexes.

A Levi sphere of $\mathcal{S}$, as defined in \cite{serre:reduct}, Section 2.1.6, is a simplicial subcomplex of $\mathcal{S}$ given by the intersection of $\mathcal{S}$ with a subvector space of $V$. Let $I$ be a subset of $\Delta$ and let 
\begin{align*}
{\bf A}_I:=(\cap_{\alpha\in I}\ker\alpha)^\circ,
\end{align*}
as usual. Under the above identification, 
\begin{align*}
\Lie(A_I)=\cap_{\alpha\in I}H_\alpha
\end{align*}
and $\mathcal{S}_I:=\Lie(A_I)\cap\mathcal{S}$ is a (standard) Levi sphere. The simplices of $\mathcal{S}_I$ parametrize the parabolic subgroups associated with $\mathcal{A}$ containing the Levi subgroup ${\bf A}_I{\bf M}_I$ of ${\bf P}_I$. The simplices of $\mathcal{S}_I$ of maximal dimension parametrize the parabolic subgroups associated with $\mathcal{A}$ such that ${\bf A}_I{\bf M}_I$ is a Levi subgroup of those parabolic subgroups. 

\section{Translates of subgroups of ${\bf M}_I$ by elements of ${\bf A}_I$}\label{s13}

Finally, we prove the following case of Conjecture \ref{mainconj}.

\begin{teo}\label{MA}
Let ${\bf P}_0$ be a minimal parabolic subgroup of $\bf G$ and let $\bf A$ be a maximal split subtorus of $\bf G$ contained in ${\bf P}_0$ such that $A$ is invariant under the Cartan involution of $G$ associated with $K$. Then Conjecture \ref{mainconj} holds when, for some $I\subseteq\Delta:=\Delta({\bf P}_0,{\bf A})$, and for each $n\in\NN$, ${\bf H}_n$ is a subgroup of ${\bf M}_I$ and $g_n\in A_I$.  
\end{teo}

\begin{proof}
Since $g_n\in A_I$, we relabel it $a_n$. Since $a_n\in A$, there exists, by Section \ref{LS}, $w_n\in W$ and $J_n\subseteq\Delta$ such that $a_n\in w_nC_{J_n}$. Since $W$ and $\Delta$ are finite, after possibly extracting a subsequence, we can and do assume that $w:=w_n$ and $J:=J_n$ are fixed. 

\begin{lem}
We have $w{\bf A}_Jw^{-1}\subseteq{\bf A}_I$.
\end{lem}

\begin{proof}
Let $\beta\in I$. Since $w\Delta$ is a set of simple $\QQ$-roots for ${\bf G}$ with respect to $\bf A$, we can write $\beta=\sum_{\alpha\in\Delta}a_\alpha w\alpha$ for some $a_\alpha\in\QQ$. Since $a_n\in wC_J\cap A_I$, we have
\begin{align*}
1=\beta(a_n)=\prod_{\alpha\in\Delta}w\alpha(a_n)^{a_\alpha}=\prod_{\alpha\notin J}w\alpha(a_n)^{a_\alpha},
\end{align*}
for every $n\in\NN$. Since $\beta$ is either positive or negative with respect to the ordering given by $w\Delta$, either $a_\alpha\geq 0$ for all $\alpha\in\Delta$, or $a_\alpha\leq 0$ for all $\alpha\in\Delta$. Therefore, since $w\alpha(a_n)>1$ for all $\alpha\notin J$, we conclude that $a_\alpha=0$ for all $\alpha\notin J$. That is, $\beta$ is contained in the $\QQ$-span of the set $wJ$, which proves the claim.
\end{proof}

Since $w\alpha(a_n)>1$, for all $n\in\NN$ and all $\alpha\notin J$, after possibly extracting a subsequence, we can and do assume that 
\begin{align*}
w\alpha(a_n)\rightarrow c_\alpha\in[1,\infty],\text{ as }n\rightarrow\infty.
\end{align*}	
Therefore, we can write $\Delta\setminus J$ as the disjoint union of two subsets $R_\infty$ and $R_0$ defined such that $\alpha\in\Delta\setminus J$ is a member of $R_0$ if and only if $c_\alpha\in[1,\infty)$ (and $\alpha\in\Delta\setminus J$ is a member of $R_\infty$ if and only if $c=\infty$).

Let ${\bf A}_\infty:={\bf A}_{J\cup R_0}$ and ${\bf A}_0:={\bf A}_{J\cup R_\infty}$. Then ${\bf A}_J={\bf A}_\infty{\bf A}_0$ and we can write $a_n\in wA_Jw^{-1}$ as $a_{n,\infty}a_{n,0}$, where $a_{n,\infty}\in w{\bf A}_\infty w^{-1}$ and $a_{n,0}\in w{\bf A}_0w^{-1}$. In particular, for every $\alpha\in R_\infty$
\begin{align*}
w\alpha(a_{n,\infty})=w\alpha(a_n)\rightarrow\infty,\text{ as }n\rightarrow\infty,
\end{align*}
and, for every $\alpha\in R_0$,
\begin{align*}
w\alpha(a_{n,0})=w\alpha(a_n)\rightarrow c_\alpha\in[1,\infty),\text{ as }n\rightarrow\infty.
\end{align*}

As in the proof of \cite{borel:arithmetic-groups}, Proposition 12.6, we can represent $w$ by an element in $K$ (which we also denote by $w$) (where we use the fact that the Weyl group of $\bf A$ is naturally a subgroup of the Weyl group of ${\bf A}_\RR$). Therefore, if we let ${\bf P}:=w{\bf P}_{J\cup R_0}w^{-1}$, then $A_{\bf P}$ is equal to $wA_\infty w^{-1}$.

Let $\Gamma_P:=\Gamma\cap P$. We conclude that, after possibly extracting a subsequence, we can and do assume that the sequence of homogeneous probability measures on $\Gamma_P\backslash P$ associated with the ${\bf H}_n$ and $a_{n,0}$ converges. Therefore, the result follows from Theorem \ref{convergencecriterion}.
\end{proof}

\bibliography{EquiDraft}
\bibliographystyle{alpha}

\end{document}